\documentclass[final]{siamltex}
\usepackage{amsmath}
\usepackage{mathtools}
\usepackage{mathrsfs}
\usepackage{amssymb}
\usepackage{amsfonts}
\usepackage{amsxtra}
\usepackage{amstext}
\usepackage{amsbsy}
\usepackage{amscd}
\usepackage{graphicx}
\usepackage{float}
\usepackage{cite}
\usepackage{xcolor}
\usepackage{srcltx} 
\usepackage{marginnote,slashbox}
\usepackage{rotating}
\definecolor{darkblue}{rgb}{0.0,0.0,0.6}
\definecolor{darkgreen}{rgb}{0.0,0.6,0.0}
\usepackage[pdftex,colorlinks=true,urlcolor=darkblue,citecolor=darkblue,linkcolor=darkblue]{hyperref}
\usepackage[utf8]{inputenc}      
\usepackage{booktabs}            
\usepackage{url}                 
\usepackage[T1]{fontenc}         
\usepackage{algorithmic}         
\usepackage[pagewise]{lineno}
\usepackage{calc}
\usepackage[notcite,notref]{showkeys}

\numberwithin{table}{section}    
\numberwithin{figure}{section}   
\numberwithin{equation}{section} 
\usepackage{my_latex_commands}

\newtheorem{assumption}[theorem]{Assumption}
\newtheorem{remark}[theorem]{Remark}

\newcommand{\dz}{\delta z}
\newcommand{\ff}{\forall\,}
\newcommand{\xx}{\widetilde x}

\newcommand{\eps}{\varepsilon}
\newcommand{\e}{\varepsilon}
\newcommand{\ee}{\quad \text{as }\e \searrow 0}

\newcommand{\yy}{\bar y}

\newcommand{\ld}{L^2(D)}

\newcommand{\hde}{H^s(D\setminus \bar E)}

\newcommand{\hd}{H^1_0(D)}

\newcommand{\gbe}{\bar g_\eps}
\newcommand{\gbec}{\{\gbe=\e z\}}
\newcommand{\gbecc}{\{x \in D \setminus \bar E: \gbe(x)=\e z\}}
\newcommand{\gbs}{\bar g_{\mathfrak{sh}}}
\renewcommand{\ss}{{\mathfrak{sh}}}
\newcommand{\gbsc}{\{\bar g_{\mathfrak{sh}}=0\}}
\newcommand{\ybs}{\bar y_{\mathfrak{sh}}}

\newcommand{\dme}{D \setminus \bar E}
\newcommand{\dbme}{\bar D \setminus E}

\newcommand{\li}{L^\infty(D)}

\renewcommand{\o}{\omega}

\newcommand{\ogg}{\O_{g}}
\newcommand{\ogs}{\O_{\gbs}}

\begin{document}

\title{necessary conditions for the optimal control of a shape optimization problem with non-smooth pde constraints
}
\date{\today}
\author{L.\,Betz\footnotemark[1]}
\renewcommand{\thefootnote}{\fnsymbol{footnote}}
\footnotetext[1]{Faculty of Mathematics,  University of W\"urzburg,  Germany}
\renewcommand{\thefootnote}{\arabic{footnote}}

\maketitle


\begin{abstract}
This paper is concerned with the derivation of necessary  conditions for the optimal shape of a design problem governed by a non-smooth PDE. The main particularity thereof is the lack of differentiability of the nonlinearity  in the state equation, which, at the same time, is solved on an unknown domain. We follow the functional variational approach introduced in \cite{pen} where the set of admissible shapes is parametrized by a large class of continuous mappings. 
 It has been recently established \cite{p1}  that each parametrization associated to an  optimal shape  is the limit of a sequence of global optima of minimization problems with convex admissible set consisting of functions. Though non-smooth, these problems allow for the derivation of an   optimality system equivalent with the first order necessary optimality condition \cite{p2}. In the present manuscript we let  the approximation parameter vanish therein. 
 The  final necessary conditions for the non-smooth shape optimization problem consist of  an adjoint equation, a limit gradient equation that features a measure concentrated on the boundary of the optimal shape and, because of the non-smoothness, an inclusion  that involves its Clarke subdifferential.\end{abstract}

  \section{Introduction}

The purpose of this paper is to derive  necessary  conditions for the optimal shape of the following 
non-smooth optimal design problem
\begin{equation}\tag{$P_\O$}\label{p_sh}
 \left.
 \begin{aligned}
  \min_{\O \in \OO, E \subset \O} \quad & \int_E (y_{\O}(x)-y_d(x))^2 \; dx+\alpha\,\int_{\O}\,dx, 
\\     \text{s.t.} \quad & 
  \begin{aligned}[t]
    -\laplace y_{\O} + \beta(y_{\O})&=f \quad \text{a.e.\,in }\O,
   \\y_{\O}&=0  \quad \text{on } \partial \O.\end{aligned} \end{aligned}
 \quad \right\}
\end{equation}
The admissible control set $\OO$ consists of unknown  subdomains of a given, fixed domain $D$. These are the so-called \textit{admissible shapes}. In this manuscript, they are generated by a class of continuous functions, see \eqref{f_s} below. The holdall domain $D\subset \R^2$ is bounded, of class $C^{1,1}$, and $f \in L^2(D)$. The symbol $-\laplace: \hd \to H^{-1}(D)$ denotes the Laplace operator in the distributional sense; note that $\hd$  is the closure of the set {$ C^\infty_c(D) $} w.r.t.\,the $H^1(D)-$norm.
The desired state $y_d$ is an $L^2-$function, which is defined on the  observation set  $E$ (Assumption \ref{assu:E}). The parameter $\alpha$ in the objective is supposed to satisfy $\alpha \geq 0$.  
The essential feature of \eqref{p_sh} is that the mapping $\beta$ appearing in the governing PDE is locally Lipschitz continuous and directionally differentiable, but \textit{not} necessarily \textit{differentiable}; see Assumption \ref{assu:stand}.\ref{it:stand2}-\ref{it:stand3}  for details. We aim at establishing an optimality system for \eqref{p_sh} consisting of an adjoint equation, a gradient inclusion and, because of the non-smoothness, a relation  that involves the Clarke subdifferential of $\beta$, see Theorem \ref{thmm0}. In this context, we also derive sign conditions for the adjoint state and the optimal state.


We underline the complexity of  \eqref{p_sh}. While shape optimization problems exhibit similarities  with optimal control problems \cite{troe}, the crucial difference and difficulty here comes from the fact   that the admissible control set consists of variable geometries \cite{pir,sz, top_book}. Such problems are highly nonconvex and  their investigation is far away from being standard.
When addressing \eqref{p_sh}, one faces an additional challenge, namely its \textit{non-smooth character}.
Even in the case of classical non-smooth control problems (where the domain is fixed), the application of traditional optimization methods is excluded,  as the standard KKT theory cannot be directly employed once the differentiability of the control-to-state map is missing. If the domain is fixed, this may be however overcome by smoothening the problem \cite{bt} or, in certain cases \cite{mcrf,paper}, by direct methods.

Let us put our work into perspective.
Shape optimization problems with non-smooth constraints have been investigated at a theoretical level mostly 
with respect to existence of optimal shapes \cite{dm1, dm2} and   sensitivity analysis \cite{ nsz, hs,sz, sok}.  While there is an increasing number of contributions   concerning  optimal shape design problems governed by VIs, see \cite{dm1,dm2, nsz, hs, ylw, lsw} and the references therein, there are no papers known to the author that address the case where the governing equation is a \textit{non-smooth PDE}. In \cite{ylw, lsw}, the authors   resort to smoothening techniques and optimality systems in qualified form are obtained just for the smoothened problem \cite{ylw}, or, for the original problem \cite{lsw}, but  only under certain rather restricitve assumptions  on the converging sequences. If smoothening is not involved, optimality conditions for the non-smooth shape optimization problem  do not involve an adjoint equation, unless linearity w.r.t.\,direction is assumed \cite{nsz}. Otherwise, these are stated just in a primal form \cite{hs}, i.e., the respective optimality condition only asserts the non-negativity of the shape derivative of the reduced objective functional in feasible directions.
As in most of the literature, the  approaches in all the aforementioned  papers (see also the references therein) are based on variations of the geometry. One of the  most common notions in this context are the 
shape derivative \cite{sz} and the topological derivative \cite{top_book}.

A more novel technique to derive optimality conditions in qualified form, where general functional variations instead of geometrical ones are involved, can be found in \cite{oc_t}. Therein, an optimal design problem governed by a linear PDE with Neumann boundary conditions is investigated. By means of the implicit parametrization theorem combined with  Hamiltonian systems \cite{it_h},  the equivalence of the shape and topology optimization problem with an optimal control problem in function space is established, provided that the set of admissible geometries is generated by a certain class of continuous functions. The respective minimization problem   is then  amenable to the derivation of optimality conditions by a  classical Lagrange multipliers approach. 
The idea that the admissible shapes are parametrized by so called shape functions turned out to be successful  in numerous papers \cite{pen, t_jde, oc_t, nt, nt2, to, mt_new2, mt_new1, hmt}.
 The variations used in all these works have no prescribed geometric form (as usual in the literature) and the  methodology therein provides a unified analytic framework allowing for both boundary and topological variations.


The recent work \cite{p1}  was a first essential step towards the derivation of  optimality systems  for the optimal shape associated to \eqref{p_sh}. 
Therein, we employed the aforementioned functional variational method, that is, we switched from the shape optimization problem \eqref{p_sh} to an optimal control problem where the admissible set consists of functions, see \eqref{p_shh} below.  The main result consisted of showing that each parametrization associated to the optimal shape of \eqref{p_sh} is the limit of global minimizers of  \eqref{p10} (Corollary \ref{cor:os}).  
Having established an optimality system for \eqref{p10}, one may  pass to the limit therein and obtain necessary conditions 
for the optimal control of \eqref{p_shh}, and then transfer these to \eqref{p_sh}. This is pricisely the scope of the present manuscript. 
It continues the investigations from  \cite{p1}, by letting the parameter $\e>0$ vanish  in the optimality conditions for the approximating minimization problem \eqref{p10}; we point out  that these were already derived in \cite{p2}, see Lemma \ref{lem:tool} and Theorem \ref{thm} below.

The reformulation of \eqref{p_sh} (for certain admissible sets  $\OO$, see \eqref{o} below) in terms of   a control problem in  function space reads as follows \begin{equation}\tag{$P$}\label{p_shh}
 \left.
 \begin{aligned}
  \min_{g \in \FF_{\mathfrak{sh}}} \quad & \int_E (y(x)-y_d(x))^2 \; dx+\alpha\,\int_{D}1-H(g)\,dx, 
\\     \text{s.t.} \quad & 
  \begin{aligned}[t]
    -\laplace y + \beta(y)&=f \quad \text{in }\O_g,
   \\y&=0  \quad \text{on } \partial \O_g, \end{aligned} \end{aligned}
 \quad \right\}
\end{equation}where $H:\R \to\{0,1\}$ is the Heaviside function, see \eqref{h} below. The set of admissible shape functions is  \begin{equation}\label{f_s}
\begin{aligned}
\FF_{\mathfrak{sh}}&:=\{g \in C^2(\bar D):g(x) \leq 0\ \forall\,x\in E, \ |\nabla g(x)|+|g(x)|>0 \ \forall\,x\in D, \\&\  \ \ \ \  \ \qquad \qquad \quad g(x)>0 \ \forall\,x\in\partial D,\   g(x)>0 \ \forall\,x \not \in \clos \ogg\} , \end{aligned}\end{equation}
see \cite[Rem.\,2.7]{p1} for more comments regarding this definition. 
For $g\in \FF_{\ss}$, $\O_g$  describes the following open subset of the holdall domain $D$:
\begin{equation}\label{def:og}
 \Omega_g:=\interior\{x \in D: g(x) \leq 0\} .
\end{equation}
We note that $\O_g$ may have many connected components. The admissible shape (domain)  that we use in the definition of $\OO$ below, see \eqref{o}, is the component that contains the subdomain $E$; its existence is guaranteed by the first inequality in \eqref{f_s}. 
Due to the particular properties of $g \in \FF_{\ss}$, the level set $\{g=0\}$ describes the smooth boundary of the subdomain defined in \eqref{def:og}, see Lemma \ref{fs} below. Moreover, each component of  $\{g=0\}$ is parametrized by the solution of a Hamiltonian system, see \eqref{ham}; in particular, $\partial \O_g$ is a finite union of closed disjoint $C^2$ curves, without self intersections, that do not intersect $\partial D$ \cite[Prop.\,2]{pen}. 
With \eqref{f_s} and \eqref{def:og} at hand, we may now define  the admissible set for the shape optimization problem \eqref{p_sh} as 
\begin{equation}\label{o}
\OO:=\{\text{the component of }\O_g \text{ that contains the set }E: g\in \FF_{\mathfrak{sh}}\}.\end{equation} 
We point out that 
 the family of admissible domains is rich, its elements may  be multiply-connected, that is,  the approach we discuss in this paper is related to topological optimization too.

The optimal control problem \eqref{p_shh} preserves the non-smoothness, that is, the standard adjoint calculus is still excluded. The admissible set consisting of functions $\FF_{\mathfrak{sh}}$ is non-convex, while the control does not appear on the right-hand side of the non-smooth PDE, but in the definition of the variable domain on which this is solved. 
 In order to deal with the challenges brought up by  \eqref{p_shh}, we use a fixed domain type method \cite{nt,nt2} and extend  the state equation on the whole reference domain $D$. By proceeding like this, one preserves the non-smooth character, and arrives, for $\e>0$ small, fixed, at the following approximating optimal control problem \cite{p1}:

\begin{equation}\tag{$P_\eps$}\label{p10}
 \left.
 \begin{aligned}
  \min_{g \in   \FF} \quad & \int_E (y(x)-y_d(x))^2 \; dx+\alpha\,\int_{D} (1-H_\eps(g))(x)\,dx+\frac{1}{2}\,\|g-\bar g_{\mathfrak{sh}}\|_{\WW}^2  \\   \text{s.t.} \quad & 
  \begin{aligned}[t]
   -\laplace y + \beta(y)+\frac{1}{\eps}H_\eps(g)y&=f +\eps g\quad \text{in }D,
   \\y&=0\quad \text{ on } \partial D, \end{aligned} 
 \end{aligned}
 \quad \right\}
\end{equation}where $\eps>0$ is a fixed parameter. The mapping $H_\e$ is the regularization of the Heaviside function, cf.\,\eqref{reg_h} below. Here, $\WW$ is the Hilbert space $\ld \cap \hde$, $s\in ( 1,2]$, endowed with the norm 
\[\|\cdot\|_{\WW}^2:=\|\cdot\|_{\ld}^2+\|\cdot\|_{H^s(D\setminus \bar E)}^2,\]
and 
\begin{equation}\label{f_tilde}
 \FF:=\{g \in \WW: g \leq 0\text{ a.e.\,in }E\}.
\end{equation}We underline the fact that the set $\FF$ from \eqref{f_tilde} is a convex subset of a Hilbert space. It is  a suitable replacement of the non-convex set $\FF_{\mathfrak{sh}}$ thanks to an essential density property \cite[Sec.\,3]{p1} that allows us to bridge the gap between these two sets.
In the objective of \eqref{p10}, $\bar g_{\mathfrak{sh}}$ is a fixed local optimum of \eqref{p_shh}. That is, $\gbs \in \FF_{\mathfrak{sh}}$, and the condition $s\leq 2$ implies  that $\gbs \in \WW$ too; this ensures that the last term in the objective of \eqref{p10} is well-defined.  As shown by Theorem \ref{thm:cor}, \eqref{p10} is constructed in such a way  that $\gbs$ arises as the limit of a sequence of local optima of \eqref{p10} \cite{barbu84}. Moreover, the resulting  approximating control problems \eqref{p10} are  amenable to the derivation of strong stationary optimality conditions, cf.\,Theorem \ref{thm}. 
For more comments regarding the particular structure of \eqref{p10}, we refer the reader to \cite[Sec.\,4]{p1}.

The paper is organized as follows. In section \ref{sec:2} we gather the standing assumptions and all the results from the previous contributions \cite{p1,p2} that are needed in the present manuscript. These   include Theorem \ref{thm:cor} and the optimality system associated to \eqref{p10}, see Lemma \ref{lem:tool}, where we intend to let the parameter $\e$ vanish. 

As it turns out, the limit analysis of the sequence $\{H_\e'(\gbe)\}$ appearing in the gradient equation  \eqref{lem:grad_f0} requires the most effort; here,  $\gbe$ denotes a local optimum of \eqref{p10}. While in the particular situation $\gbe=\gbs$ the respective passage to the limit  is rather standard (Remark \ref{rem:conv}), the general setting we deal with is far more involved.
 It involves the description of trajectories corresponding to $\gbe$ via Hamiltonian systems, cf.\,Proposition \ref{prop:per}. Section \ref{sec:grad} is entirely dedicated to their study, including a careful analysis of their  properties  as $\e$ approaches 0. 
The findings therein are applicable for general sequences of mappings that converge in $C^2$ towards an arbitrary mapping from $\FF_{\ss}$ (not necessarily $\gbs$). 
 It turns out that the elements of the converging sequence inherit the essential properties of the limit function on $D \setminus \bar E$, see Lemma \ref{lem:g_e}. To each one of them we can associate subdomains, that, for $\e$ sufficiently small, preserve the topology (the number of holes and components is the same as in the case of $\O_{\gbs} \setminus \bar E$, see Remark \ref{rem:top}).
 Corollary \ref{int_conv0}  establishes the convergence  behaviour of curvilinear integrals associated to optima of the approximating problem \eqref{p10}. This is the final key contribution of section \ref{sec:grad} that makes the passage to the limit in the third term in \eqref{eq_d-e} possible. 
  
 In the final section \ref{pas} we prove our main result by  first deriving an optimality system for  local optima of \eqref{p_shh}. This is contained in Theorem \ref{thmm}. It consists of an adjoint equation, a subdifferential inclusion featuring the Clarke subdifferential of the non-smoothness and an inclusion that corresponds to the limit gradient equation. Moreover, we make use of the fact that we need the standing Assumption \ref{yd_f0} anyway for the approximation property in Theorem \ref{thm:cor}. The alternative requirements in \eqref{f01} and \eqref{f00} then allow us to derive sign conditions for the optimal state and the adjoint state, cf.\,\eqref{p>=0} and \eqref{p<=0}. Finally, we transfer the necessary conditions from Theorem  \ref{thmm} to the shape optimization problem \eqref{p_shh} and arrive at the main result of the present manuscript, which is Theorem \ref{thmm0}.

  \section{Standing assumptions and preliminary results}\label{sec:2}
In this section we  collect the standing assumptions and those results from the previous contributions \cite{p1,p2} that are needed in the present paper.

  \begin{assumption}[The non-smoothness]\label{assu:stand}
 \begin{enumerate}
 \item\label{it:stand2}The  function  $\beta: \R \to \R$ is   
monotone increasing and {locally} Lipschitz continuous {in the following sense: F}or all $M>0$, there exists 
  a constant $L_M>0$ such that
  \begin{equation*}
   |\beta(z_1) - \beta(z_2)| \leq L_M |z_1 - z_2| \quad \forall\, z_1, z_2 \in [-M,M] .
  \end{equation*}
  \item\label{it:stand3}The mapping $\beta$ is directionally differentiable at every point, i.e., 
  \begin{equation*}
   \Big|\frac{\beta(z + \tau \,\dz) - \beta(z)}{\tau} - \beta'(z;\dz)\Big| \stackrel{\tau \searrow 0}{\longrightarrow} 0 \quad \forall \, z,\dz \in \R.
  \end{equation*}
  \end{enumerate}
\end{assumption}

In the rest of the paper, one tacitly supposes that Assumption \ref{assu:stand} is always fulfilled without mentioning it every time. By Assumption \ref{assu:stand}.\ref{it:stand2}, it is straight forward to see that the Nemytskii operator $\beta:\li  \to \li$ is well-defined. Moreover, this is Lipschitz continuous on bounded sets in the following sense:  for every  $M > 0$, 
  there exists $L_M > 0$ so that
  \begin{equation}\label{eq:flip}
   \|\beta(y_1) - \beta(y_2)\|_{L^q(D)} \leq L_M \, \|y_1 - y_2\|_{L^q(D)} \quad \forall\, y_1,y_2 \in \clos{B_{\li}(0,M)},\ \forall\, 1\leq q \leq \infty.
  \end{equation}
 
\begin{assumption}[The observation set]\label{assu:E}
The  set $E\subset D$  is a  subdomain with boundary of measure zero. Moreover, $\dist(\clos E,\partial D)>0,$ where \[\dist(\clos E,\partial D):=\min_{(x_1,x_2)\in \clos E \times \partial D} \dist(x_1,x_2).\]\end{assumption}

\begin{definition}The non-linearity $H_\eps:\R \to [0,1]$ is defined as follows 
 \begin{equation}\label{reg_h}
H_\eps(v):=\left\{\begin{aligned}0,&\quad \text{if }v\leq 0,
\\\frac{v^2(3\eps-2v)}{\eps^3},&\quad \text{if }v\in (0,\eps),
\\1,&\quad \text{if }v\geq \eps.
\end{aligned}\right.\end{equation}
\end{definition}
\\We note that                                            
$H_\e$ is the regularization of the 
Heaviside function $H:\R \to [0,1]$, which is defined as
 \begin{equation}\label{h}
H(v):=\left\{\begin{aligned}0,&\quad \text{if }v\leq 0,
\\1,&\quad \text{if }v>0.
\end{aligned}\right.\end{equation}

For later purposes, we compute here its derivative:
 \begin{equation}\label{reg_h'}
H_\eps'(v):=\left\{\begin{aligned}0,&\quad \text{if }v\leq 0,
\\\frac{6v(\eps-v)}{\eps^3},&\quad \text{if }v\in (0,\eps),
\\0,&\quad \text{if }v\geq \eps.
\end{aligned}\right.\end{equation}
Now we continue with some essential results that were established previously \cite{p1,p2}.
\begin{lemma}[Properties of admissible shapes and $\O_g$, {\cite[Lem.\,2.9]{p1}}]\label{fs}
Let $g \in \FF_{\mathfrak{sh}}$ and denote by $\O \in \OO$ the relevant component of $\O_g,$ that is, the component that contains $E$. Then,
\begin{enumerate}
\item \label{class} $ \O \text{ is a domain of class }C^2$;
\item \label{class2} $\partial \O_g=\{x \in D: g(x) = 0\}$ and $\O_g=\{x \in D: g(x) < 0\};$    
\item \label{ls} $\mu\{x \in D: g(x) = 0\}=0$.
\end{enumerate}
\end{lemma}

\begin{lemma}[ {\cite[Lem.\,4.11]{p1}}]\label{lem:cp}
Let $\{g_\e\} \subset \FF$ and $g \in \FF_{\mathfrak{sh}}$ so that 
\[g_\e \to g \quad \text{in }\ld \cap L^\infty(D \setminus \bar E)\quad \text{as }\e \searrow 0.\]
Then, for each compact subset $K$ of $\ogg$, 
there exists $\e_0>0$, independent of $x$, so that
\begin{equation}\label{g_k}
g_\e \leq 0 \quad  \ae K,\ \forall\,\e\in(0,\e_0].\end{equation}
\end{lemma}
 In order to be able to  state the main results of \cite{p1,p2}, we need to recall some definitions.

\begin{definition}[The control-to-state map associated to the state equation in \eqref{p_shh}, {\cite[Def.\,4.9]{p1}}]\label{S}
We define
\begin{equation}
\SS:g \in \FF_{\mathfrak{sh}} \mapsto y_g \in  H^1_0(\O_g) \cap H^2(\O_g), \end{equation}
where $y_g$ solves the state equation in \eqref{p_shh} on the component  of $\O_g$ containing $E$.\end{definition}

\begin{definition}[Local optimum of \eqref{p_shh}]\label{def_sh}
We say that $\bar g_{\mathfrak{sh}} \in \FF_{\mathfrak{sh}}$ is locally optimal for \eqref{p_shh} in the $\ld$ sense, if there exists $r>0$ such that 
\begin{equation}\label{loc_opt_shh}
\JJ(\bar g_{\mathfrak{sh}}) \leq \JJ(g) \quad \ff g \in \FF_{\mathfrak{sh}} \text{ with }\|g-\bar g_{\mathfrak{sh}}\|_{\ld}\leq  r,
\end{equation}where $$\JJ(g):=\int_E (\SS( g)(x)-y_d(x))^2 \; dx+\alpha\,\int_{D} (1-H(g))(x)\,dx.$$\end{definition}

\begin{definition}
The control-to-state map associated to the state equation in \eqref{p10} is denoted by 
 \[S_\e:\ld \to  H_0^1(D) \cap H^2(D).\]

 \end{definition}
For each $g\in \ld$, it holds  
\begin{equation}\label{lem:S}
\|S_\e(g)\|_{H^1_0(D) \cap C(\bar D)} \leq c_1+c_2\,\|g\|_{\ld},
\end{equation}
where $c_1,c_2>0$ are independent of $\eps$, $\beta$, and  $g$. This result has been established in \cite[Lem.\,4.5]{p1}.

\begin{definition}[Local optimum of the approximating control problem, {\cite[Def.\,5.1]{p1}}]\label{def_e}
Let $\e>0$ be fixed and $\bar g_{\mathfrak{sh}} \in \FF_{\mathfrak{sh}}$.
We say that $\bar g_\eps \in \FF$ is locally optimal for \eqref{p10} in the $\ld$ sense, if there exists $r>0$ such that 
\begin{equation}\label{loc_opt_e}
j_\eps(\bar g_\eps) \leq j_\eps(g) \quad \ff g \in \FF \text{ with }\|g-\bar g_\e\|_{\ld}\leq  r,
\end{equation}where $$j_\eps(g):=\int_E (S_\eps( g)(x)-y_d(x))^2 \; dx+\alpha\,\int_{D} (1-H_\eps(g))(x)\,dx+\frac{1}{2}\,\|g-\bar g_{\mathfrak{sh}}\|_{\WW}^2.$$ \end{definition}

\begin{assumption}\label{yd_f0}
For the desired state we require   $ y_d \in H^1_0(E)$ with $\laplace y_d \in L^2(E).$
Moreover, \begin{equation}\label{f01}
f \geq \chi_E(-\laplace y_d+\beta(y_d))+\chi_{D\setminus E}\beta(0)\quad \ae D,\end{equation}or, alternatively, 
\begin{equation}\label{f00}
f \leq \chi_E(-\laplace y_d+\beta(y_d))+\chi_{D\setminus E}\beta(0)\quad \ae D.\end{equation}
\end{assumption}
\begin{theorem}[{Approximation property, {\cite[Thm.\,5.2]{p1}}}]\label{thm:cor}
Let $\bar g_{\mathfrak{sh}} \in \FF_{\mathfrak{sh}}$ be a local minimizer of \eqref{p_shh} in the sense of Definition \ref{def_sh}. Then, there exists a sequence of local minimizers $\{\bar g_\eps\}$ of \eqref{p10} such that 
\[\bar g_\eps \to \bar g_{\mathfrak{sh}} \quad \text{in }\WW, \text{ as }\eps \searrow 0.\]{Moreover,}\[S_\eps(\gbe) \weakly \SS(\bar g_{\mathfrak{sh}})\quad \text{in }\hd, \text{ as }\eps \searrow 0.\]
\end{theorem}

\begin{corollary}[Approximation of  the optimal shape, {\cite[Cor.\,5.4]{p1}}]\label{cor:os}
Let $\O^\star \in \OO$ be an optimal shape for \eqref{p_sh}. Then, for each  $g^\star \in \FF_{\mathfrak{sh}}$ with $\O_{g^\star}=\O^\star$, there exists  
a sequence of global minimizers $\{\bar g_\eps\}$ of \eqref{p10} such that 
\[\bar g_\eps \to g^\star \quad \text{in }\WW, \text{ as }\eps \searrow 0.\]
{Moreover,}\[S_\eps(\gbe) \weakly \SS(g^\star)\quad \text{in }\hd, \text{ as }\eps \searrow 0.\]
\end{corollary}

We continue by recalling the main results from \cite{p2}.
\begin{assumption}\label{assu:reg0}
For each nonsmooth point $z \in \R$ of the mapping $\beta:\R \to \R$, there exists $\delta>0$,  so that $\beta$
  is convex or  concave  on the interval $[z-\delta, z+\delta]$. Moreover, the set of non-smooth points $\NN$ is countable and $\beta$ is continuously differentiable outside the intervals $[z-\delta_z/4, z+\delta_z/4]$, where $z \in \NN$, and $\delta_z>0$ is the associated radius of convexity/concavity.\end{assumption}

\begin{lemma}[Optimality system for the approximating control problem, {\cite[Lem.\,3.4]{p2}}]\label{lem:tool} 
Let $\bar g_\e \in \FF$ be locally optimal for \eqref{p10} in the sense of Definition \ref{def_e} and denote by $\bar y_\e \in \hd \cap H^2(D)$ its associated state. Then, there exists an 
 adjoint state $p_\e \in \hd \cap H^2(D)$ and a multiplier $\zeta_\e \in \li$ so that 
 the optimality system is satisfied:
 \begin{subequations} \label{eq:lem}  \begin{gather}
-\laplace p_\e+\zeta_\e p_\e+\frac{1}{\eps}H_\eps(\bar g_\e)p_\e=2\chi_E(\bar y_\e-y_d) \ \text{ a.e.\,in }D, \quad p_\e=0 \text{ on }\partial D, \label{lem:adj}
\\\zeta_\e(x) \in [\min\{\beta_-'(\yy_\e(x)),\beta_+'(\yy_\e(x))\},\max\{\beta_-'(\yy_\e(x)),\beta_+'(\yy_\e(x))\}] \quad \text{a.e.\,in }D,\label{lem:clarke0}
\\( p_\eps[\eps-\frac{1}{\eps}H_\eps'(\bar g_\eps)\bar y_\eps]-\alpha H_\eps'( \bar g_\eps),h-\bar g_\eps)_{\ld}+(\bar g_\eps-\bar g_{\mathfrak{sh}},h-\bar g_\eps)_{\WW}\geq 0\quad \forall\,h\in \FF,\label{lem:grad_f0}
   \end{gather}\end{subequations}
   where, for an arbitrary $z\in \R$, the left and right-sided derivative of $\beta: \R \to \R$  are defined through
$\beta'_-(z) := -\beta'(z;-1)$ and $ \beta'_+(z) :=\beta'(z;1)$, respectively. Moreover, 
\eqref{lem:grad_f0} is equivalent to the following two relations
\begin{equation}\label{cone}
p_\e[\eps-\frac{1}{\eps}H_\eps'(\bar g_\e)\bar y_\e]-\alpha H_\eps'( \bar g_\e)+\bar g_\e-\bar g_{\mathfrak{sh}} \leq 0 \quad \ae E,
\end{equation}
 \begin{equation}\label{eq_d-e}
  ( p_\e[\eps-\frac{1}{\eps}H_\eps'(\bar g_\e)\bar y_\e]-\alpha H_\eps'( \bar g_\e),\phi)_{L^2(D\setminus E)}+(\bar g_\e-\bar g_{\mathfrak{sh}},\phi)_{L^2(D\setminus E)}+(\bar g_\e-\bar g_{\mathfrak{sh}},\phi)_{\hde}= 0  \end{equation}
   \end{lemma}for all $\phi \in \hde.$\begin{proof}The optimality system \eqref{lem:adj}-\eqref{lem:grad_f0} has been established in \cite[Lem.\,3.4]{p2}. The last assertion is due to the arguments employed at the beginning of the proof of \cite[Thm.\,3.14]{p2}.\end{proof}

\begin{theorem}[Strong stationarity, {\cite[Thm.\,3.14, Corollary 4.2]{p2}}]\label{thm}
Assume that only convexity is allowed in Assumption \ref{assu:reg0} and that   
\begin{equation}\label{esssup}\esssup_{D\setminus \bar E} f <\beta(0).\end{equation}
Let $\{\bar g_\e\} \subset \FF$ be a sequence of local optima of \eqref{p10} that converges in $\WW$ and denote by $\{\bar y_\e\} \subset \hd \cap H^2(D)$ the sequence of associated states.  Then, there exists $\e_0>0$ so that for each $\e\in (0,\e_0]$ there exists  an 
 adjoint state $p_\e \in \hd \cap H^2(D)$ and a multiplier $\zeta_\e \in \li$ so that 
 the optimality system is satisfied:
\begin{subequations}  \begin{gather}
-\laplace p_\eps+\zeta_\eps p_\eps+\frac{1}{\eps}H_\eps(\bar g_\eps)p_\eps=2\chi_E(\bar y_\eps-y_d) \text{ a.e.\ in }D, \quad p_\eps=0 \text{ on }\partial D, 
\\\zeta_\eps(x) \in [\beta_-'(\yy_\eps(x)),\beta_+'(\yy_\eps(x))] \text{ a.e.\,in } D,
\\p_\e(x) \leq 0 \ \text{ a.e.\,where  } \yy_\eps(x) \in \NN , \label{pe}
\\( p_\eps[\eps-\frac{1}{\eps}H_\eps'(\bar g_\eps)\bar y_\eps]-\alpha H_\eps'( \bar g_\eps),h-\bar g_\eps)_{\ld}+(\bar g_\eps-\bar g_{\mathfrak{sh}},h-\bar g_\eps)_{\WW}\geq 0\quad \forall\,h\in \FF.
   \end{gather}\end{subequations}
\end{theorem}

\begin{remark}\label{rem}
The strict inequality in \eqref{esssup} implies that \eqref{f01} cannot be true. This means that, for Theorem \ref{thm} to be valid, we need to work under the second alternative requirement in the standing Assumption \ref{yd_f0}. \\We also remark that the optimality system in Theorem \ref{thm} is indeed strong stationary, i.e., equivalent with the first order necessary optimality condition associated to the local minimum of $j_\e$ from Definition \ref{def_e} (see \cite[Thm.\,3.18]{p2}). \end{remark}

\section{Properties and limit behaviour of curves $\gbec$ for $z \in [0,1]$ fixed}\label{sec:grad}

In this section we carry out the preparatory work for passing to the limit in the gradient equation \eqref{lem:grad_f0}. 
Let us point out from the very beginning that all the results in this section stay true if $\{\gbe\}$ is replaced by  an arbitrary sequence of $C^2(\bar D \setminus E)$-functions satisfying the convergence \eqref{c1} below and if $\gbs$ is an arbitrary element of the set $\FF_{\ss}$. Moreover, instead of $\gbec$ we may consider curves $\{\gbe=z_\e\}$ where $\{z_\e\}$ is a real-valued sequence approaching $0$ as $\e \searrow 0$. 

In all what follows, $\gbs \in \FF_{\ss}$ is a local minimizer of \eqref{p_shh} and $\{\bar g_\eps\} \subset \FF$ is a sequence of local optimizers of \eqref{p10} such that 
\[\bar g_\eps \to \bar g_{\mathfrak{sh}} \quad \text{in }\WW \text{ as }\eps \searrow 0.\]Note that its existence  is guaranteed by Theorem \ref{thm:cor}.
Moreover, we fix $s>3$ so that the following embedding is true \cite[p.\,88]{kaballo}
\begin{equation}\label{emb}
H^s(D \setminus \bar E) \embed C^2(\bar D \setminus  E).\end{equation} 
This implies 
\begin{equation}\label{c1}\bar g_\eps \to \bar g_{\mathfrak{sh}} \quad \text{in }C^2(\bar D \setminus  E) \quad \text{ as }\eps \searrow 0,\end{equation}
which will be crucial for the proof of the limit gradient equation \eqref{grad} below. 
Recall that  in introduction we imposed an upper bound on $s$, i.e., $s \leq 2,$ but this was only needed to ensure that  the last term in the objective of \eqref{p10} is well-defined. In order to guarantee that $\gbs \in \FF_{\mathfrak{sh}} \subset \WW$ is still true, we require: 
\begin{assumption}\label{gbs}
There exists $s>3$ so that $\gbs \in H^s(D \setminus \bar E)$.
\end{assumption}

\begin{remark} 
The requirement in Assumption \ref{gbs} is not restrictive, as one may define the set $\FF_{\mathfrak{sh}}$ via $C^{3,\gamma}(\bar D)$-mappings, $\gamma \in (0,1]$, instead of $C^2(\bar D)$-functions; see the comments in \cite[Sec.\,2.1]{p1}. The boundaries of the admissible shapes then have $C^3$ regularity (Lemma \ref{fs}) and  the previous investigations from \cite{p1,p2} remain unaffected by the change in the definition of the set of admissible parametrizations. Of course, the  drawback here is that the admissible  set of  shapes $\OO$, see \eqref{o}, becomes smaller.   \end{remark}

In the rest of the section, Assumption \ref{gbs} is tacitly assumed without mentioning it every time.
Thanks to the uniform convergence \eqref{c1}, the most important properties of $\gbs \in \FF_{\mathfrak{sh}}$, cf.\,\eqref{f_s}, are inherited by $\gbe$ (for $\e$ small enough).

\begin{lemma}\label{lem:g_e}
There exists $\delta>0$, dependent only on $D, E$ and $\gbs$, so that
\begin{equation}\label{eq:gbe0}
|\nabla \gbe|+|\gbe| \geq \delta \quad \text{in }\bar D \setminus  E,\end{equation}
\begin{equation}\label{eq:gbe0+}\gbe \geq \delta \quad \text{on }\partial D ,\end{equation}
for all $\e\in(0,\e_0],$ where $\e_0$ is small, dependent only on $\delta$ and the given data. In particular,  \begin{equation}\label{eq1}
  \min_{\gbe^{-1}[0,\eps]\cap(\bar D \setminus  E)} |\nabla \gbe|\geq \delta/2.\end{equation}
\end{lemma}
\begin{proof}In view of Assumption \ref{assu:E}, the set $\bar D \setminus  E$ is compact and since $\gbs \in \FF_{\ss},$ see \eqref{f_s}, we have 
\[|\nabla \gbs|+|\gbs| \geq 2\delta \quad \text{in }\bar D \setminus  E\]for some $\delta>0.$ 
The first desired assertion then follows from  \eqref{c1}. The second assertion can be concluded analogously, as $\partial D$ is compact and $\gbs >0$ on $\partial D,$ cf.\,\eqref{f_s}. The estimate \eqref{eq1} is an immediate consequence of \eqref{eq:gbe0} with $0 <\e \leq \delta/2.$
\end{proof}

 In the rest of the section, $z \in [0,1]$ is an arbitrary but fixed value. We keep the dependency on $z$ of the involved quantities in mind the whole time. Thanks to Lemma \ref{lem:g_e}, we know that there exists $\delta >0$ so that 
\begin{equation}\label{eq:gbe}
|\nabla \gbe|+|\gbe-\e z| \geq \delta/2 \quad \text{in }\bar D \setminus  E,\end{equation}
\begin{equation}\label{eq:gbe+}\gbe-\e z  \geq \delta/2 \quad \text{on }\partial D ,\end{equation}
for all $\e\in(0,\e_0],$ where $\e_0$ is small, dependent only on $\delta$ and the given data.
Note that $\e_0$ and $\delta$ are both independent of $z$. The above relations are of outmost importance as they allow us to conclude that the level sets $\{x \in \dbme:\gbe(x)=\e z\}$ have the same (useful) properties as $\partial \O_{\gbs}$ (Lemma \ref{fs}), see Proposition \ref{prop:per} below. Since the convergence \eqref{c1} is satisfied only on $\bar D \setminus E$ (not on the whole $\bar D$),  we have to make sure that the curves  $\{x \in \dbme:\gbe(x)=\e z\}$ remain in $D \setminus \bar E$ the whole time, i.e., they do not "cut" $\partial E$ (nor $\partial D$, which is already true due to \eqref{eq:gbe+}). This aspect  is guaranteed under the following 

\begin{assumption}\label{assu:d}
In the rest of the paper, we assume that
\[\gbs<0 \quad \text{in }\clos E.\]
\end{assumption}
\begin{remark}
The condition in Assumption \ref{assu:d} is equivalent to \[\partial E \cap \partial \O_{\gbs}=\emptyset,\]in view of Lemma \ref{fs}.\ref{class2}. In the context of optimal control,  the requirement that the boundary of $\O_{\gbs}$ does not touch the boundary of $E$ simply corresponds to  the fact that  "the constraint is not touched by the optimal shape".
\end{remark}

From now on, Assumption \ref{assu:d} is tacitly supposed to be true without mentioning it every time.  In all what follows, the set $\{x \in \dbme:\gbe(x)=\e z\}$ is simply denoted by $\gbec$ and instead of $\partial \O_{\gbs}$ we often write $\gbsc$ (Lemma \ref{fs}.\ref{class2}). As the next result shows, these level sets are indeed included in the interior of $\dbme,$ i.e., $\dme$ (cf.\,Assumption \ref{assu:E}). 


\begin{proposition}[{\cite[Prop.\,2]{pen}}]\label{prop:per}
Suppose that Assumption \ref{assu:d} is true.
Let $\e>0$ be fixed and small. Then, the level set $\gbec$  is a finite union of disjoint closed $C^2$ curves, without self intersections and not intersecting $\partial D$ nor $\partial E$. This is parametrized by the solution of  the Hamiltonian system \begin{equation}\label{ham_e}
 \left\{
 \begin{aligned}
    x_{\e,1} '(t)&=-\frac{\partial \gbe}{\partial x_2} (x_{\e,1} (t),x_{\e,2} (t)) ,
    \\    x_{\e,2} '(t)&=\frac{\partial \gbe}{\partial x_1}(x_{\e,1} (t),x_{\e,2} (t)) ,
   \\   (x_{\e,1} (0),x_{\e,2} (0))&=x^{\e}_0 \in \{\gbe=\e z\}  \end{aligned}
 \quad \right.
\end{equation} when some initial point $x_0^\e \in D \setminus \bar E$ with $\gbe(x^{\e}_0)=\e z$ is chosen on each component. The same assertion holds true for $\gbsc$. Each of its components is parametrized by the solution of
  \begin{equation}\label{ham}
 \left\{
 \begin{aligned}
   \bar x_{1} '(t)&=-\frac{\partial \gbs}{\partial x_2} (\bar x_{1} (t),\bar x_{2} (t)) ,
    \\   \bar  x_{2} '(t)&=\frac{\partial \gbs}{\partial x_1}(\bar x_{1} (t),\bar x_{2} (t)) ,
   \\   (\bar x_{1} (0),\bar x_{2} (0))&=\bar x_0 \in \{\gbs=0\}, \end{aligned}
 \quad \right.
\end{equation}where $\bar x_0 \in D \setminus \bar E$ with $\gbs(\bar x_0)=0$ is an initial point on the respective component. 
\end{proposition}
\begin{proof}
We start by noticing that Assumption \ref{assu:d} guarantees the existence of a point $x_0^\e \in \dme$ so that $\gbe(x_0^\e)=\e z$, provided that $\e>0$ is small enough.
 This is due to \eqref{eq:gbe+}, the fact that $\gbe \leq \frac{1}{2}\max_{\partial E}\gbs<0$ on $\partial E$, see \eqref{c1}, and the continuity of $\gbe$ on the connected set $\dbme$, cf.\,\eqref{emb}.
This also implies  that $\gbec \subset \dme.$
By \eqref{eq:gbe} and the regularity of $\gbe$, see \eqref{emb}, \eqref{ham_e} admits a unique $C^2$ solution $x_\e:(-\infty,\infty) \to \dme$, thanks to the implicit function theorem. 
Then, by the arguments from the proof of {\cite[Prop.\,2]{pen}}, one can show the remaining of desired result. The statement regarding $\gbs$ is due to $\gbs \in \FF_{\mathfrak{sh}}$ and \cite[Prop.\,2]{pen}, see also Lemma \ref{fs}.
\end{proof}

\begin{remark}
In the treatment of \eqref{p_sh}, we are interested only in the component of $\O_{\gbs}$ that contains the observation set $E$, cf.\,\eqref{o}. 
However, the boundary of this particular component may be a union of  several closed curves (disjoint, without self intersections, contained in $\gbsc$), that is, $\O_{\gbs}$ may have holes. Thanks to Proposition \ref{prop:per}, see also \cite[Prop.\,2]{pen}, their number is finite. With a little abuse of notation, we use for each one of them the same notation, that is, $\{\gbs=0\}$. Later, when we integrate over the whole boundary of the relevant component of $ \O_{\gbs}$ (see Corollary \ref{int_conv0} below), we will have a finite sum of similar terms associated to
{each} component of $\partial \O_{\gbs}$, by fixing some initial condition in \eqref{ham} on each closed curve. The same observation applies to $\{\gbe=\e z\}$. 
\end{remark}
\subsection{The existence of a unique approximating curve}\label{subsec1}
Next we focus on showing that to each component $\{\gbs=0\}$ of the boundary of $\O_{\gbs}$ we may associate exactly one closed curve $\{\gbe=\e z\}$, provided that $\e>0$ is small enough. The existence of such an \textit{approximating} curve is established in  Lemma \ref{ex}  below, while its uniqueness follows from {Lemma \ref{lem:unique}}.

 \begin{lemma}\label{lem:l}
For each $\lambda>0$ there exists $\e_0>0$ so that
\[\{x \in \bar D\setminus E: \gbe(x) \in [-\e,\e]\} \subset V_{\lambda} \quad \forall\,\e\in (0,\e_0],\]
where \[V_{\lambda}:=\{x \in \bar D\setminus E:  \dist(x,\partial \O_{\gbs})<\lambda\}.\]
\end{lemma}

  \begin{proof}Let $\lambda>0$ be arbitrary but fixed. We follow the ideas of the proof of \cite[Prop.\,3]{pen}. We recall that,  according to  Lemma \ref{fs}.\ref{class2}, it holds $\partial \O_{\gbs}=\{x \in D: \gbs(x) = 0\}$. 
Since $\bar D \setminus E \setminus V_{\lambda}$ is compact and since $\gbs$ is continuous, we have 
\[|\gbs| \geq \delta \quad \text{in } \bar D \setminus E \setminus V_{\lambda}\]for some $\delta>0$ dependent only on $\lambda$ and the given data. In view of \eqref{c1}, this implies
\[|\gbe| \geq \delta/2 \quad \text{in } \bar D \setminus E \setminus V_{\lambda} \quad \forall\,\e\in (0,\e_0],\]for some $\e_0>0$ dependent only on $\lambda$ and the given data. By choosing $\e_0<\delta/2,$ we infer the first desired assertion. 
\end{proof}

\begin{lemma}[Existence of an approximating curve]\label{ex}
Let \[d:=\min_{\omega_1 \neq \omega_2, \\ \o_1,\omega_2 \in \MM} \dist({\o_1}, {\o_2})>0\]
and $\MM:=\{\omega: \omega \text{ is a component of }\partial \O_{\gbs}\}\cup \partial E.$
If Assumption \ref{assu:d} is satisfied, then, for each $\widetilde \lambda \in(0,d/3]$ and each  component $\{\gbs=0\}$ of $\partial \O_{\gbs}$ there exists $\e_0>0$ so that 
\begin{equation}\label{inclusion}
\{ \gbe=\e z\} \subset \{x \in  D\setminus \bar E:  \dist(x,\gbsc)<\widetilde \lambda\} \quad \forall\,\e\in (0,\e_0],\end{equation}where $\{ \gbe=\e z\}$ is a closed curve as in Proposition \ref{prop:per}.
In particular,
 we have the convergence 
\begin{equation}\label{h_d_c}
d_{\HH}(\{ \gbe=\e z\} , \gbsc)\to 0 \ee,\end{equation}
where $d_{\HH}(\{ \gbe=\e z\} , \gbsc)$ denotes the Hausdorff-Pompeiu  distance between the compact sets $\{ \gbe=\e z\}$ and $ \gbsc$, i.e.,
\begin{equation}\begin{aligned}
&d_{\HH}(\{ \gbe=\e z\} , \gbsc)\\ &\quad:=\max \{\max_{x \in \gbsc}\dist(x,\{ \gbe=\e z\}),\max_{\widetilde x \in \{ \gbe=\e z\}}\dist(\widetilde x, \gbsc)\}.\end{aligned}\end{equation}
  \end{lemma}
\begin{proof}
To show the inclusion \eqref{inclusion}, it suffices to prove that there exists    a point $x_0^\e$ in $\dme$ with $\dist(x_0^\e,\gbsc)<\widetilde \lambda$ that satisfies $\gbe(x_0^\e)=\e z$, where $\e$ is smaller than some fixed value $\e_0>0$. The desired assertion then follows by Proposition \ref{prop:per} and from Lemma \ref{lem:l} with $\lambda:=\widetilde \lambda$. Since $\gbs$ is continuous, there exists  $\gamma>0$ so that 
\[\gbs \geq \gamma \text{ on }K_1 \text{ and }\gbs \leq -\gamma \text{ on }K_2,\]
where $K_1:= \{x \in  \bar D\setminus E:  \dist(x,\gbsc)=\widetilde \lambda/2\} \setminus \O_{\gbs}$
and $K_2:= \{x \in  \bar D\setminus E:  \dist(x,\gbsc)=\widetilde \lambda/2\} \cap  \O_{\gbs}$; note that these sets  are compact and non-empty, since $\gbs \in \FF_{\mathfrak{sh}}$, cf.\,the last inequality in \eqref{f_s} and Lemma \ref{fs}.\ref{class2}. Hence, by \eqref{c1}, we have for $\e$ small enough, dependent on $\widetilde \lambda$ and on $\gbs$, that 
\[\gbe \geq \e \text{ on }K_1 \text{ and }\gbe \leq 0 \text{ on }K_2.\]From here we deduce that 
 $\{x \in \bar D\setminus E: \dist(x,\gbsc)\leq \widetilde \lambda/2, \gbe(x)= \e z\}$ is non-empty, as a consequence of the continuity of $\gbe$ on the connected set $\{x \in \bar D\setminus E: \dist(x,\gbsc)\leq \widetilde \lambda/2\}$.
  The convergence \eqref{h_d_c} then follows from the definition of the Hausdorff-Pompeiu distance and \eqref{inclusion}.\end{proof}

Note that Lemma \ref{ex} does not say anything about the uniqueness of the approximating curve $\gbec$; that is, for the same $\e$ one may have two different approximating closed curves $\gbec$ contained in the set $\{x \in  D\setminus \bar E:  \dist(x,\gbsc)<\widetilde \lambda\}$. However,  we show next that if $\widetilde \lambda$ is chosen even smaller, then  there exists exactly one approximating curve
 in the $\widetilde \lambda-$region surrounding $\gbsc$.

 \begin{lemma}[Uniqueness of the approximating curve]\label{lem:unique}
Suppose that Assumption \ref{assu:d} holds true and let $\{\gbs=0\}$ be a component of $\partial \O_{\gbs}$. 
Then, there exists $\lambda>0$ and  $\e_0(\lambda)>0$ so that for each $\e \in (0,\e_0(\lambda)]$ there is exactly one closed curve $\{ \gbe=\e z\}$ satisfying
\[\{ \gbe=\e z\} \subset \{x \in  D\setminus \bar E:  \dist(x,\gbsc)< \lambda\} .\]  
In particular, it holds
 \begin{equation}\label{grad_u}
|\nabla \gbe(x)| \geq  \delta/2 \quad \forall\,x \in  D\setminus \bar E \text{ with }  \dist(x,\gbsc)< \lambda, \quad \forall\,\eps\in (0,\e_0(\lambda)],\end{equation}where $\delta>0$ is dependent only on $D,E$ and $\gbs$. Thus, for each $\e \in (0,\e_0(\lambda)],$ the mapping $\gbe$  cannot have local extremum points in $\{x \in  D\setminus \bar E:  \dist(x,\gbsc)< \lambda\}$.
\end{lemma}

\begin{proof}
Let $\delta>0$ be given by Lemma \ref{lem:g_e}; that is, $\delta>0$ is dependent only on $D,E$ and $\gbs$. This means that \begin{equation}\label{eq:gbe01}
|\nabla \gbe|+|\gbe| \geq \delta \quad \text{in }\bar D \setminus  E,\end{equation}
for all $\e\in(0,\e_0],$ where $\e_0$ is small, dependent only on $\delta$ and the given data.  Set \[\lambda:=\min\{\delta/(2L_{\gbs}),\dist(\gbsc,\partial D), \dist(\gbsc,\partial E)\},\] where $L_{\gbs} >0$ is the Lipschitz constant of $\gbs.$ Note that $\lambda>0,$ by Assumption \ref{assu:d} and since $\gbs>0$ on $\partial D$.
We define $\MM:= \{x \in  D\setminus \bar E:  \dist(x,\gbsc)\leq \lambda/2\}$. Then, for each $x \in \MM$, there exists $\xx \in \gbsc$ so that $|x-\xx|\leq \frac{\lambda}{2}$. By the Lipschitz continuity of $\gbs$,  we have 
\[|\gbs(x)|=| \gbs(x)- \gbs(\xx)|\leq L_{\gbs} |x-\xx|\leq L_{\gbs} \frac{\lambda}{2}.\]

In view of the definition of $\lambda,$ the set $\MM$ is compact. From  \eqref{c1} we further deduce 
\[|\gbe(x)|\leq L_{\gbs} \lambda \leq \frac{\delta}{2} \quad \forall\, x \in \MM,\]for $\e>0$ small enough dependent on $\lambda$ and the given data; in particular, $\e>0$ is independent of $x$. The inequality \eqref{eq:gbe01} now implies that there exists $\e_0(\lambda)>0$ small enough, dependent only on $\lambda$ and the given data, with
 \begin{equation}
|\nabla \gbe| \geq  \delta/2 \quad \text{in }\MM \quad \forall\,\eps\in (0,\e_0(\lambda)].\end{equation}After rescaling $\lambda,$ this proves \eqref{grad_u} and the last statement. Then, the first desired assertion follows by a contradiction argument.
%
\end{proof}

\begin{remark}\label{rem:top}
The uniqueness result in Lemma \ref{lem:unique} has some essential implications. It basically says that, for $\e>0$ small enough, inside the $\lambda-$ region of $\gbs$, the mapping $\gbe-\e z$ has a strict sign in the interior of the curve $\gbec$ and this is different from its strict sign in the exterior of $\gbec$. If this were not the case,  an additional curve $\gbec$ would be found in the proximity of $\gbsc$, in view of Proposition \ref{prop:per}; however, this contradicts the assertion of Lemma \ref{lem:unique}.

Therefrom we also deduce that the number of components of $\gbecc$ is the same as the number of components of $\partial \O_{\gbs}$, i.e., it does not depend on $\e$ (for $\e$ small enough). Recall that  level curves $\gbec$ cannot be found outside the $\lambda-$ region of $\gbs$, as a consequence of Lemma \ref{lem:l} (provided that $\lambda$ and $\e$ take appropriate small values).
This allows us to conclude that $\gbe>\e z$ outside
 the closure of the following subdomain of class $C^2$
\[\widehat \O_{\gbe}:=\{x \in D \setminus \bar E:\gbe(x)<\e z\},\] see Proposition \ref{prop:per} and Lemma \ref{fs}. Similarly to $\O_{\gbs}$, the open set $\widehat \O_{\gbe}$ may have  holes, in which case their number is equal to the number of holes of $\O_{\gbs}$.  In other words, the topology of $\O_{\gbs}$ is preserved by the approximating domains $\widehat \O_{\gbe}$ for $\e$ small enough. Similar observations can be found in \cite[Prop.\,3.1]{oc_t}, in the special case $\gbe=\gbs+\e \kappa,$ where the function $\kappa$ belongs to a set similar to $\FF_{\ss}$. Finally, let us mention that the sequence of subdomains $\{\widehat \O_{\gbe}\}_\e$ parametrically converges to $\O_{\gbs}\setminus \bar E \ee$ \cite[Def.\,A.3.3]{nst_book}.
\end{remark}

\subsection{Convergence properties}\label{subsec2}

In all what follows, we are interested only in the components of $\{\gbs=0\}$ whose (finite) union equals the boundary of  $\O_{\gbs} \in \OO$, see \eqref{o}.   In view of Lemma \ref{lem:unique}, we know that, starting with some $\e>0,$ there is exactly one component of $\{\gbe=\e z\} $ in the $\lambda-$region surrounding $\gbsc$, provided that $\lambda$ is small enough. Note that the choice of $\e$ is dependent only on  $\lambda$ and the given data.

\begin{proposition}\label{lem:vec}
Suppose that Assumption \ref{assu:d} holds true.
Let $\{\gbs=0\}$ be a component of $\partial \O_{\gbs}$ with  periodicity  $T_{\ss}$ and let $\{\gbe=\e z\}$ be its approximating curve with periodicity $T_{\e}$, that is, $\e \in (0,\e_0(\lambda)],$ where $\lambda,\e_0(\lambda)>0$ are as in Lemma \ref{lem:unique}.
Then,
\begin{itemize}
\item[(i)]  For each $B>0$, it holds  $   \| x_{\e}  - \bar x\|_{C^1([0,B];\R^2)} \to 0 $ \text{as }$\e \searrow 0,$ where $x_\e$ and $\bar x$ are  the trajectories of the Hamiltonian systems \eqref{ham_e} and \eqref{ham} associated to $\{\gbe=\e z\}$ and $\{\gbs=0\}$, respectively;
\item[(ii)] $T_\e \to T_{\ss}$ as $\e \searrow 0.$
\end{itemize}
\end{proposition}

\begin{proof}Fix $\bar x_0 \in \gbsc$ and choose $x_0^\e\in \gbec$ as the  projection of $\bar x_0$ on $\gbec.$ Then, thanks to \eqref{h_d_c}, it holds  \begin{equation}\label{init}
 |x_0^\e - \bar x_0 |\to 0 \quad \text{as }\e \to 0.\end{equation}

Now, consider the Hamiltonian systems \eqref{ham} and \eqref{ham_e} with initial conditions $\bar x_0$ and $x_0^\e$ as above. 
We abbreviate 
\[G_\e:=(-\frac{\partial \gbe}{\partial x_2} , \frac{\partial \gbe}{\partial x_1}): D \setminus \bar E \to \R^2, \quad \bar G:=(-\frac{\partial \gbs}{\partial x_2} , \frac{\partial \gbs}{\partial x_1}): D \setminus \bar E \to \R^2. \]
For all $t \in \R_{+}$, we have 
\begin{equation}
 \begin{aligned}
    | x_{\e} (t)-\bar x(t)|&\leq |x^{\e}_0-\bar x_0|+\int_0^t |G_\e(x_{\e} (s)) \pm 
    G_\e(\bar x (s))-\bar G(\bar x (s))|\,ds ,
\\&\leq |x^{\e}_0-\bar x_0|+c \int_0^t  | x_{\e} (s)-\bar x(s)|\,ds+\int_0^{t}   |G_\e(\bar x (s))-\bar G(\bar x (s))|\,ds,   \end{aligned}
\end{equation}
  where we used that $G_\e$ is Lipschitz continuous with Lipschitz constant $c>0,$ independent of $\e,$ see \eqref{c1}.
  Applying Gronwall's inequality yields 
  \begin{equation}\label{gronw}
 \begin{aligned}
 | x_{\e} (t)-\bar x(t)| &\leq C( |x^{\e}_0-\bar x_0|+\int_0^{t}   |G_\e(\bar x (s))-\bar G(\bar x (s))|\,ds)
 \\& \leq C_\e+tc_\e \quad \forall\,t \in \R_{+},\end{aligned}
\end{equation}where \[C_\e, c_\e \to 0 \ee.\]This is a consequence of \eqref{init} and \eqref{c1}; here we use that $s \mapsto \bar x(s)$ is contained in $D 
\setminus \bar E$.
Going back to \eqref{ham} and \eqref{ham_e}, where this time we use \eqref{gronw} and estimate as above, we get
  \begin{equation}\label{gronw'}
 \begin{aligned}
 | x_{\e} '(t)-\bar x'(t)| &= |G_\e(x_{\e} (t)) \pm 
    G_\e(\bar x (t))-\bar G(\bar x (t))|
    \\&\leq c\,|x_{\e}(t)-\bar x (t)|+  |G_\e(\bar x (t))-\bar G(\bar x (t))|
 \\&\leq c(C_\e+t\,c_\e)+c_\e  \quad \forall\,t \in \R_{+}, \end{aligned}
\end{equation}where \[C_\e, c_\e \to 0 \ee.\]
The estimates \eqref{gronw} and \eqref{gronw'} now give the first desired assertion (i). 

Let us now show (ii). By arguing as in the proof of {\cite[Prop.\,2.6]{to}} we see that $\{T_\e\}$ is uniformly bounded from above by $j\,T_{\ss}, j <2$. To see that $\{T_\e\}$ is uniformly bounded from below by {a strictly positive constant}, we notice that \[\ell(x_\e)\geq c,\]where $c>0$ is dependent only on the length of the curve $\bar x$; this is a consequence of \eqref{gronw}. 
From \eqref{ham_e} we infer that
\[\ell(x_\e)=\int_0^{T_\e} |\nabla \gbe(x_\e(s))| \,ds \geq c,\]cf.\,e.g.\cite{pressley}. Thus, in view of \eqref{c1},  there exists $C>0$, independent of $\e$, so that
\[T_\e \geq C.\]
All the above considerations then imply
\begin{equation}\label{T}
T_\e \to T \in (0,2 T_\ss) \ee,
\end{equation}
up to a subsequence. From \eqref{gronw}, we deduce that for all $t \in \R_{+}$ it holds 
 \begin{equation}
 \begin{aligned}
0=\lim_{\e \searrow 0} | x_{\e} (t)-\bar x(t)|=\lim_{\e \searrow 0} | x_{\e} (t+T_\e)-\bar x(t+T_\e)| =\lim_{\e \searrow 0} | x_{\e} (t)-\bar x(t+T_\e)|.\end{aligned}
\end{equation}Note that in the last identity we employed the fact that $s \mapsto x_\e(s)$ is periodic with periodicity $T_\e$.
The above convergence implies
\[\lim_{\e \searrow 0} | \bar x (t)-\bar x(t+T_\e)|=0 \quad \ff t \in \R_{+}.\]The continuity of $\bar x$ together with \eqref{T} then yields $ \bar x (t)=\bar x(t+T) \ \ff t \in \R_{+}.$ Hence, $T=T_{\ss},$ as $\bar x$ is periodic with periodicity $T_\ss$ and since $T \in (0,2 T_\ss),$ cf.\,\eqref{T}. Since the arguments apply for each convergent subsequence of $\{T_\e\}$, we can thus infer (i) from \eqref{T}. The proof is now complete.
  \end{proof}

  \begin{proposition}\label{int_conv}
  Suppose that Assumption \ref{assu:d} is true.
  Let $\{h_\e\}\subset C(\dbme;\R) $ be a sequence  that converges uniformly to $h$ on $\dbme$. Then,
  \begin{equation}\label{conv_h}
   \int_{\{\gbe=\e z\}} \frac{h_\e}{|\nabla \gbe|} \,d \xi \to  \int_{\{\gbs=0\}}\frac{h}{|\nabla \gbs|}\,d \xi \ee,\end{equation}
   where $\{\gbs=0\}\subset \dme$ is an arbitrary component of $\partial \O_{\gbs}$ and $\gbec \subset \dme$ is its approximating curve.
  Moreover,
  \begin{equation}\label{bound_h} \int_{\{\gbe=\e z\}} \frac{h_\e}{|\nabla \gbe|} \,d \xi \leq c,\end{equation}
where $c>0$ is independent of $\e$ and $z$.
  \end{proposition}\begin{proof}
  We fix $\lambda>0$ and $\e_0(\lambda)>0$  as in Lemma \ref{lem:unique}.
  Throughout this proof, $\e \in (0,\e_0(\lambda)]$.  We begin by noticing that the integrands in \eqref{conv_h} are well-defined in view of \eqref{eq:gbe} and since $\gbs \in \FF_{\ss}$, see \eqref{f_s}.
As a result of Proposition \ref{prop:per}, we have 
\begin{equation}\label{conv_h1}
\begin{aligned}
  \int_{\gbe=\e z} \frac{h_\e}{|\nabla \gbe|} \,d \xi  =\int_0^{T_{\e}}\frac{h_\e(x_\e(s))}{|\nabla \gbe(x_\e(s))|}|x_\e'(s)| \,ds
 =\int_0^{T_{\e}}{h_\e(x_\e(s))} \,ds,  \end{aligned}
  \end{equation}
where $T_\e>0$ is the periodicity of $\gbec.$ Further, we denote by $\bar x:(-\infty,\infty) \to \dme$ the solution of the Hamiltonian system \eqref{ham} associated to the closed curve $\gbsc$ with periodicity $T_{\ss}>0.$
Since $h_\e \to h$ in $C(\dbme)$, by assumption, and in view of Proposition \ref{lem:vec}, it holds
  \begin{equation}\label{conv_h2}\begin{aligned}
 & \int_0^{T_{\e}}{h_\e(x_\e(s))} \,ds-\int_0^{T_{\ss}}{h(\bar x(s))} \,ds \\&\quad =
   \int_0^{\min\{T_{\e},T_{\ss}\}}[{h_\e(x_\e(s))}-{h(\bar x(s))} ] \,ds
   \\&\qquad +\int_{\min\{T_{\e},T_{\ss}\}}^{T_\e}{h_\e(x_\e(s))} \,ds -\int_{\min\{T_{\e},T_{\ss}\}}^{T_{\ss}}{h(\bar x(s))} \,ds
   \\&\quad \leq C \|h_\e -h\|_{C(\dbme)}+ \int_0^{C}|{h(x_\e(s))}-{h(\bar x(s))} | \,ds
   \\&\qquad +(T_\e-{\min\{T_{\e},T_{\ss}\}}) \|h_\e\|_{C(\dbme)}
   \\&\qquad +(T_{\ss}-{\min\{T_{\e},T_{\ss}\}}) \|h\|_{C(\dbme)}
   \\&\quad \to 0 \ee,
   \end{aligned}\end{equation}
   where $C>0$ is  a constant independent of $\e.$
   By Proposition \ref{prop:per}, cf.\,\eqref{ham}, we deduce 
    \begin{equation}
   \int_{\gbs=0}\frac{h}{|\nabla \gbs|}\,d \xi   =\int_0^{T_{\ss}}\frac{h(\bar x(s))}{|\nabla \gbs(\bar x(s))|}|\bar x'(s)| \,ds
 =\int_0^{T_{\ss}}{h(\bar x(s))} \,ds,    \end{equation}
and the desired convergence \eqref{conv_h} now follows from  \eqref{conv_h1} and  \eqref{conv_h2}. Finally, we see that \eqref{bound_h} is a consequence of \eqref{conv_h1}, Proposition \ref{lem:vec} and the uniform boundedness  of $\|h_\e\|_{C(\dbme)}$ w.r.t.\,$\e$.
  \end{proof}
  
By building a finite sum over the terms associated to
each component of $\{x \in D \setminus \bar E: \gbe(x)=\e z\} $, we arrive at the following
consequence of Proposition \ref{int_conv}:
  \begin{corollary}\label{int_conv0}
  Suppose that Assumption \ref{assu:d} is true.
  Let $\{h_\e\}\subset C(\dbme;\R) $ be a sequence  that converges uniformly to $h$ on $\dbme$. Then,
  \begin{equation}\label{conv_h0}
   \int_{\{x \in D \setminus \bar E: \gbe(x)=\e z\} } \frac{h_\e}{|\nabla \gbe|} \,d \xi \to  \int_{\partial \O_{\gbs}}\frac{h}{|\nabla \gbs|}\,d \xi \ee.\end{equation}
  Moreover,
  \begin{equation}\label{bound_h0} \int_{\{x \in D \setminus \bar E: \gbe(x)=\e z\} } \frac{h_\e}{|\nabla \gbe|} \,d \xi \leq c,\end{equation}
where $c>0$ is independent of $\e$ and $z$.
  \end{corollary}
\begin{remark}
As already mentioned, all the results in this section stay true if $\{\gbe\}$ is replaced by  an arbitrary sequence of $C^2(\bar D \setminus E)$-functions satisfying the convergence \eqref{c1}. Moreover, instead of $\gbec$ we may consider curves $\{\gbe=z_\e\}$ where $\{z_\e\}$ is a real-valued sequence approaching $0$ as $\e \searrow 0$. To the best of our knowledge, similar findings to the one proven in this subsection can  be found in the literature only for particular cases. For instance, the result in Proposition \ref{lem:vec} resembles the one in \cite[Prop.\,2.6]{to} if $\gbe:=\gbs+\e \kappa,$ where  the mapping $\kappa$ belongs to a set similar to $\FF_{\ss}.$ 
For the convergence \eqref{conv_h} in the special case when $h_\e=h$ and $\gbe=\gbs$, we refer to \cite[App.\,6]{cm}.
\end{remark}

\section{Passage to the limit $\e \searrow 0$}\label{pas}
In this section, we turn our attention to the original control problem \eqref{p_shh}, which we recall  for convenience:
\begin{equation}\tag{$P$}\label{p}
 \left.
 \begin{aligned}
  \min_{g \in \FF_{\mathfrak{sh}}} \quad & \int_E (y(x)-y_d(x))^2 \; dx+\alpha\,\int_{D}1-H(g)\,dx, 
\\     \text{s.t.} \quad & 
  \begin{aligned}[t]
    -\laplace y + \beta(y)&=f \quad \text{in }\O_g,
   \\y&=0  \quad \text{on } \partial \O_g. \end{aligned} \end{aligned}
 \quad \right\}
\end{equation}
Our next goal is to establish a  limit optimality system for \eqref{p}, by letting $\e \searrow 0$  in the optimality conditions from Lemma \ref{lem:tool}. Then, we will return to the shape optimization problem \eqref{p_sh} and establish a final optimality system for the optimal shape (Theorem \ref{thmm0} below).

\begin{theorem}[Limit optimality system]\label{thmm}
Let $\bar g_{\mathfrak{sh}} \in \FF_{\mathfrak{sh}}$ be a local optimum in the $\ld$ sense of the control problem \eqref{p_shh} with associated state $\yy_\mathfrak{sh}:=\SS(\bar g_{\mathfrak{sh}})\in H^1_0(\O_{\bar g_{\mathfrak{sh}}}) \cap H^2(\O_{\bar g_{\mathfrak{sh}}})$.  If {Assumptions \ref{gbs} and  \ref{assu:d} hold true}, then there exists an adjoint state $p  \in  H^1_0(\O_{\bar g_{\mathfrak{sh}}})\cap H^2(\O_{\bar g_{\mathfrak{sh}}})$ and a non-negative multiplier $\zeta \in \li$ such that 
\begin{subequations} \label{eq:thmm0}  \begin{gather}
-\laplace p+\zeta p=2\chi_E(\yy_\mathfrak{sh}-y_d) \text{ a.e.\ in }\O_{\bar g_{\mathfrak{sh}}}, \quad p=0 \text{ on }\partial \O_{\bar g_{\mathfrak{sh}}}, \label{adj_s0}
\\\zeta(x) \in [\min\{\beta_-'(\yy_\mathfrak{sh}(x)),\beta_+'(\yy_\mathfrak{sh}(x))\},\max\{\beta_-'(\yy_\mathfrak{sh}(x)),\beta_+'(\yy_\mathfrak{sh}(x))\}] \text{ a.e.\,in } D,\label{clarke0_s}
\\\alpha \delta_{\gbs} \in \mathcal{Q}(\gbs,\ybs,p),\label{grad} 
\end{gather}\end{subequations}
where $\delta_{\gbs} \in C(\bar  D \setminus  E)^\star$ corresponds to a measure concentrated on $\partial \O_{\gbs}$, and it is  defined as 
   \begin{equation}\label{delta}
 \delta_{\gbs}(\phi):=    \int_{\gbs=0} \frac{  \phi}{|\nabla \gbs|}\,d \xi \quad \forall\,\phi \in C(\bar  D \setminus  E)\end{equation}
and 
 \begin{equation}\label{q}
 \begin{aligned}
 \mathcal Q & (\gbs,\ybs,p) :=\{w \in C(\bar  D \setminus  E)^\star: -\frac{1}{\e}\mathfrak h_\e  \widetilde y_\e \widetilde p_\e\weakly w \text{ in }H^s( D \setminus \bar E)^\star, \\&\text{ where } \mathfrak h_\e\weakly  \delta_{\gbs} \  \text{in }H^s( D \setminus \bar E)^\star,\  \widetilde y_\e \weakly \ybs \  \text{in }H_0^1(D),\ \widetilde p_\e \weakly p \  \text{in }H_0^1(D) \}.\end{aligned}
 \end{equation}
Moreover, if \eqref{f01} is true in the standing Assumption \ref{yd_f0}, then
\begin{equation}\label{p>=0}
\ybs \geq y_d \text{ in }E, \quad \ybs \geq 0 \text{ in }\O_{\gbs} \setminus E, \quad p \geq 0  \text{ in }\O_{\gbs}.
\end{equation}Otherwise, i.e.,
if  \eqref{f00} is true, then
\begin{equation}\label{p<=0}
\ybs \leq y_d \text{ in }E, \quad \ybs \leq 0 \text{ in }\O_{\gbs} \setminus E, \quad p \leq 0   \text{ in }\O_{\gbs}.
\end{equation}
\end{theorem}

\begin{proof}(I) \textit{Uniform boundedness.} By Theorem \ref{thm:cor}, there exists a sequence $\{\bar g_\eps\} \subset \FF$ of local optima of \eqref{p10} with 
\begin{equation}\label{conv_g_e}
\bar g_\eps \to \bar g_{\mathfrak{sh}} \quad \text{in }\WW, \quad \text{ as }\eps \searrow 0
\end{equation}
{and}
\begin{equation}\label{conv_y_e}
\bar y_\e:=S_\eps(\gbe) \weakly \yy_\mathfrak{sh} \text{ in }\hd, \quad \text{ as }\eps \searrow 0.
\end{equation}Let $\e>0$ be arbitrary, but fixed. According to Lemma \ref{lem:tool}, there exists an adjoint state 
$p_\e \in \hd \cap H^2(D)$ and a multiplier $\zeta_\e \in \li$ so that it holds 
\begin{subequations} \label{eq:thmm}  \begin{gather}
-\laplace p_\eps+\zeta_\eps p_\eps+\frac{1}{\eps}H_\eps(\bar g_\eps)p_\eps=2\chi_E(\bar y_\eps-y_d) \text{ a.e.\ in }D, \quad p_\eps=0 \text{ on }\partial D, \label{adj_s}
\\\zeta_\eps(x) \in [\min\{\beta_-'(\yy_\eps(x)),\beta_+'(\yy_\eps(x))\},\max\{\beta_-'(\yy_\eps(x)),\beta_+'(\yy_\eps(x))\}] \text{ a.e.\,in } D,\label{clarke0_s0}
\\( p_\eps[\eps-\frac{1}{\eps}H_\eps'(\bar g_\eps)\bar y_\eps]-\alpha H_\eps'( \bar g_\eps),h-\bar g_\eps)_{\ld}+(\bar g_\eps-\bar g_{\mathfrak{sh}},h-\bar g_\eps)_{\WW}\geq 0\quad \forall\,h\in \FF. \label{grad_f0_s}
   \end{gather}\end{subequations}Since $\beta$ is monotonically increasing, we infer from \eqref{clarke0_s0} that $\zeta_\e \geq 0\ \ae D.$ As $H_\e(\gbe)\geq  0$ (see \eqref{reg_h}), we can thus conclude from \eqref{adj_s} and \eqref{conv_y_e} that  
\[\|p_\e\|_{\hd \cap \li} \leq c,\]where $c>0$ is independent of $\eps;$ see e.g.\,\cite[Thm.\,4.8]{troe}.
Thus, there exists  $p \in \hd \cap \li$ so that 
\begin{equation}\label{pw}
p_\e \weakly p \quad \text{in }\hd \quad \text{ as }\eps \searrow 0, \end{equation}
\begin{equation}\label{pw0}
p_\e \weakly^\star p \quad \text{in }\li \quad \text{ as }\eps \searrow 0, \end{equation} for a (not relabeled) subsequence.
Further, by  \eqref{eq:flip}, we have 
\begin{equation}\label{bet}
\|\beta'(\bar y_\eps;\pm 1)\|_{\li}\leq L_{M},\end{equation}
where $M>0$ is such that $\|\bar y_\eps\|_{\li} \leq M.$ Thanks to \eqref{lem:S} and \eqref{conv_g_e}, $M>0$ can be chosen independent of $\eps$. Hence, in light of \eqref{clarke0_s} and \eqref{bet}, we can extract a (not relabeled) subsequence of $\{\zeta_\eps\}$ with
\begin{equation}\label{conv_xi}\zeta_\e \weakly^\star \zeta \quad \text{in }\li \quad \text{ as }\eps \searrow 0.\end{equation}

(II) \textit{Convergence analysis.}
\\(i) \textit{Adjoint equation.} 
Next, we want to pass to the limit in \eqref{adj_s}. We follow  the ideas of the proof of \cite[Lem.\,4.11]{p1}. Testing \eqref{adj_s} with $\eps p_\eps$, and using the uniform boundedness of $\{p_\e\}, \{\zeta_\e\}$ and $\{\bar y_\e\}$, implies
  \begin{equation}\label{hgp}
  \int_D H_\eps(\gbe)p_\eps^2 \,dx \to 0 \text{ as }\eps \searrow 0.
  \end{equation}
 As a result of \eqref{conv_g_e}, we further  have 
\[\mu\{x \in D:\bar g_{\mathfrak{sh}}<0 \text{ and }\bar g_\e \geq 0\} \to 0 \quad \text{as }\e \searrow 0,\]
\[\mu\{x \in D:\bar g_{\mathfrak{sh}}>0 \text{ and }\bar g_\e-\e \leq 0\} \to 0 \quad \text{as }\e \searrow 0.\]
Therefore,
\[\lim_{\e \to 0}  \int_{\{\bar g_{\mathfrak{sh}}<0, \bar g_\e\geq 0\}} H_\eps(\bar g_\e)p_\eps^2 \,dx+\lim_{\e \to 0}  \int_{\{\gbs>0,\bar g_\e\leq \e\}} H_\eps(\bar g_\e)p_\eps^2 \,dx=0.\]
  Thus, by  \eqref{hgp}, Lemma \ref{fs}, \eqref{reg_h} and \eqref{pw} combined with the compact embedding $H^1(D) \embed \embed \ld$, one deduces
\begin{equation}\label{h_y}
\begin{aligned}
0&=\lim_{\e \to 0} \int_D H_\eps(\bar g_\e)p_\eps^2 \,dx 
\\&=\lim_{\e \to 0}  \int_{\{\gbs<0, \bar g_\e<0\}} H_\eps(\bar g_\e)p_\eps^2 \,dx+\lim_{\e \to 0}  \int_{\{\gbs>0,\bar g_\e>\e\}} H_\eps(\bar g_\e)p_\eps^2 \,dx
\\&=\lim_{\e \to 0}  \int_{\{\gbs>0,\bar g_\e>\e\}} p_\eps^2 \,dx=\lim_{\e \to 0}  \int_{\{\gbs>0\}} p_\eps^2 \,dx
=\lim_{\e \to 0}  \int_{D\setminus \ogs} p_\eps^2 \,dx
\\&=\int_{D\setminus  \O_{\gbs}}  p^2 \,dx .\end{aligned}
\end{equation}
%
Hence, $p=0\ \ae D\setminus   \O_{\gbs}$. In view of \cite[Thm.\,2.3]{t_c} (applied for the component of $\ogs$ that contains the set $E$), we have   
\begin{equation}\label{pho1}
p \in H^1_0( \ogs).
\end{equation}
Note that here we use that $\ogs \in \OO$ is a domain of class $C$ (even $C^2$, see Lemma \ref{fs}.\ref{class}).
Testing \eqref{adj_s} with a fixed $\phi \in C_c^\infty( \O_{\bar g_{\mathfrak{sh}}})$, $\ogs \in \OO$, further leads to 
$$\int_{ \O_{\bar g_{\mathfrak{sh}}}} \nabla p_\eps \nabla \phi \,dx +\int_{ \O_{\bar g_{\mathfrak{sh}}}} \zeta_\eps p_\eps \phi \,dx +\int_{ \O_{\bar g_{\mathfrak{sh}}}}\frac{1}{\eps}H_\eps(\bar g_\eps)p_\eps \phi \,dx=\int_{ \O_{\bar g_{\mathfrak{sh}}}} 2\chi_E(\bar y_\eps-y_d)\phi \,dx.$$
Since $\WW \embed \ld \cap L^\infty(D \setminus \bar E)$, one obtains  \[\gbe \to \bar g_{\mathfrak{sh}} \text{ in }\ld \cap L^\infty(D \setminus \bar E)\quad \text{as }\e \searrow 0,\]in light of \eqref{conv_g_e}. We note that  there exists a compact subset $\tilde K$ of $\O_{\bar g_{\mathfrak{sh}}}$ so that 
$\phi =0$ in $\ogs \setminus \tilde K$. 
Hence, by Lemma \ref{lem:cp} and \eqref{reg_h}, the third term in the above variational identity vanishes for $\e>0$ small enough, independent of $x$ (dependent on $\tilde K$, and thus on $\phi$). 
 Passing to the limit $\eps \searrow 0$, where one relies on \eqref{pw}, \eqref{conv_xi}, and \eqref{conv_y_e} then results in
$$\int_{ \O_{\bar g_{\mathfrak{sh}}}} \nabla p \nabla \phi \,dx +\int_{ \O_{\bar g_{\mathfrak{sh}}}} \zeta p \phi \,dx =\int_{ \O_{\bar g_{\mathfrak{sh}}}} 2\chi_E(\yy_\mathfrak{sh}-y_d)\phi \,dx.$$Since $p \in H^1_0(\O_{\bar g_{\mathfrak{sh}}})$, cf.\,\eqref{pho1}, this shows \eqref{adj_s0}. Note that the $H^2$ regularity of $p$ is a consequence of the fact that $\ogs \in \OO$ is a domain of class  $C^2$, see Lemma \ref{fs}.\ref{class}.

(ii) \textit{{Clarke subdifferential}.} Next we want to prove \eqref{clarke0_s}. From \cite[Thm.\,7.3.2b, Prop.\,7.3.9d), Thm.\,7.3.12]{schirotzek}
and 
\eqref{clarke0_s0} we deduce that
\begin{equation}\label{intt}
\int_D \zeta_\e(x)v\phi(x)\,dx \leq  \int_D \beta^\circ(\yy_\e(x);v)\phi(x)\,dx
\end{equation}
for all $v\in \R$ and $\phi \in C_c^\infty(D), \phi \geq 0.$
Moreover, 
\begin{equation}\label{betta}
\beta^\circ(z;v) =\left\{
\begin{aligned}
\beta'(z)(v) &\quad \text{if }\beta \text{ is continuously differentiable at } z,
 \\   \beta'(z;v) &\quad \text{if }\beta \text{ is convex around } z,
\\-\beta'(z;-v) &\quad \text{if }\beta \text{ is concave around } z
\end{aligned}\right.\end{equation}for all $z,v \in \R$, see \cite[Thm.\,7.3.2b]{schirotzek}. In step (I) of the proof, we have shown that \eqref{bet} is true with $M>0$  independent of $\eps$. Hence, due to \eqref{betta}, the integrand on the right hand side  in \eqref{intt} is uniformly bounded. This permits us to apply the generalized Fatou's Lemma (see e.g.\,\cite[p.\,151]{elstrodt}), from which we  infer
\[
  \limsup_{\e \to 0} \int_D \beta^\circ(\yy_\e(x);v)\phi(x)\,dx
\leq   \int_D \limsup_{\e \to 0} \beta^\circ(\yy_\e(x);v)\phi(x)\,dx.
\]
By \eqref{conv_xi} and  the upper semicontinuity of $ \beta^\circ(\cdot;\cdot)$ \cite[Prop.\,7.3.8a]{schirotzek}  in combination with \eqref{conv_y_e}, we then arrive at
\begin{equation}\label{zeta0}
\begin{aligned}
\int_D \zeta(x)v\phi(x)\,dx =
\lim_{\e \to 0} \int_D \zeta_\e(x)v\phi(x)\,dx &\leq  \limsup_{\e \to 0} \int_D \beta^\circ(\yy_\e(x);v)\phi(x)\,dx
\\&\leq   \int_D \limsup_{\e \to 0} \beta^\circ(\yy_\e(x);v)\phi(x)\,dx
\\&\leq   \int_D  \beta^\circ(\yy_\mathfrak{sh}(x);v)\phi(x)\,dx.
\end{aligned}\end{equation}
Now, the fundamental lemma of calculus of variations and \eqref{betta} yield \eqref{clarke0_s}. 
Note that, in light of Assumption \ref{assu:stand}.\ref{it:stand2},  we can deduce from here the non-negativity of $\zeta$.

(iii) \textit{Gradient equation.}
Next we focus on establishing  \eqref{grad}. According to Lemma \ref{lem:tool},  \eqref{grad_f0_s} implies  \eqref{cone} and \eqref{eq_d-e}. In view of $\bar g_\e \leq 0 $ a.e.\ in $E$ and \eqref{reg_h'}, 
letting $\e \searrow 0$ in \eqref{cone} yields  $0\leq 0$, cf.\ \eqref{conv_g_e} and \eqref{pw}. Hence we only need to pass to the limit in \eqref{eq_d-e}, which reads 
  \begin{equation}\label{eq_d-ee}
  ( p_\e[\eps-\frac{1}{\eps}H_\eps'(\bar g_\e)\bar y_\e]-\alpha H_\eps'( \bar g_\e),\phi)_{L^2(D\setminus E)}+(\bar g_\e-\bar g_{\mathfrak{sh}},\phi)_{L^2(D \setminus \bar E)}+(\bar g_\e-\bar g_{\mathfrak{sh}},\phi)_{\hde}= 0  \end{equation}
for all $\phi \in H^s(D \setminus \bar E )$.

%
%

We start by analyzing   the third term in \eqref{eq_d-ee}.
As a result of \eqref{reg_h'}, it holds   \begin{equation}\label{eq2}
  H_\eps'(v)=\frac{1}{\eps}\Psi (\frac{v}{\e}),\end{equation} where   \begin{equation}\label{psi}
\Psi(v):=\left\{\begin{aligned}0,&\quad \text{if }v\leq 0,
\\-6v^2+6v,&\quad \text{if }v\in (0,1),
\\0,&\quad \text{if }v\geq 1.
\end{aligned}\right.\end{equation}Note that $\Psi:\R \to [0,3/2]$ is a continuous function with
\begin{equation}\label{int_psi}
\int_{0}^{1}\Psi(z)\,dz=1.\end{equation}

 Let $\e>0$ be arbitrary but fixed and small enough so that \eqref{eq1} is true. We make use of the fact that $H_\e'$ vanishes outside the interval $[0,\eps]$ (cf.\,\eqref{reg_h'}) and \eqref{eq1}, which allow us to apply the co-area formula, see for instance \cite[Prop.\,3, Sec.\,3.4.2]{eg}. In light of \eqref{eq2}, we  obtain
 \begin{equation}\label{he'ge}
 \begin{aligned}
( H_\eps'( \gbe),\phi)_{L^2(D \setminus \bar E)}&=\int_{\R^2} H_\eps'( \gbe)\chi_{D \setminus \bar E} \phi  \,dx
\\&=\int_0^\e \int_{\gbe=t} \frac{H_\eps'( \gbe)\chi_{D \setminus \bar E}  \phi}{|\nabla \gbe|}\,d \xi \,dt
\\&=\int_0^\e \frac{1}{\eps} \Psi (\frac{t}{\e})  \int_{\gbe=t} \frac{\chi_{D \setminus \bar E} \phi}{|\nabla \gbe|}\,d \xi \,dt
\\&=\int_0^1  \Psi (z)  \underbrace{\int_{\gbe= \e z} \frac{\chi_{D \setminus \bar E} \phi}{|\nabla \gbe|}\,d \xi}_{=:\eta_\e(z)} \,d z.
 \end{aligned}
  \end{equation}
Next we address  the pointwise convergence of the sequence $\{\eta_\e\} \subset C([0,1];\R).$ Note that \[\eta_\e:[0,1]\ni z \mapsto  \int_{\gbe=\e z } \frac{\chi_{D \setminus \bar E} \phi}{|\nabla \gbe|}\,d \xi \in \R\] is indeed continuous; this follows by arguing in the exact same way as in the proof of Proposition \ref{int_conv}. 
We point out that in the definition of $\eta_\e$ we build the integral over the entire level set $\{x \in \dme:\gbe(x)=\e z\}$; recall that this is a finite union of closed disjoint curves without self intersections (Proposition \ref{prop:per}).
By Corollary \ref{int_conv0}, one has
\[ \eta_\e(z) \to \int_{\gbs=0} \frac{\chi_{D \setminus \bar E}  \phi}{|\nabla \gbs|}\,d \xi  \quad \text{as }\e \to 0,\  \forall\,z \in [0,1].\]
Thanks to \eqref{bound_h} and since $\Psi$ is continuous, see \eqref{psi}, we conclude  \begin{equation}\label{iiia0}   \int_0^1  \Psi (z)  \eta_\e(z) \,d z   \overset{ \e \searrow 0}{\to} \int_0^1  \Psi (z) \,d z \int_{\gbs=0} \frac{\chi_{D \setminus \bar E}  \phi}{|\nabla \gbs|}\,d \xi = \int_{\gbs=0} \frac{\chi_{D \setminus \bar E}  \phi}{|\nabla \gbs|}\,d \xi.\end{equation}
Note that the above identity is due to \eqref{int_psi}. Therefore, by \eqref{he'ge}, we immediately get
 \begin{equation}\label{eq_d-ee0}
  (H_\eps'(\bar g_\e),\phi)_{L^2(D\setminus E)}
  \to  \int_{\gbs=0} \frac{\chi_{D \setminus \bar E}  \phi}{|\nabla \gbs|}\,d \xi =\delta_{\gbs}(\phi) \ee  \end{equation}
for all $\phi \in H^s(D \setminus \bar E )$; see \eqref{delta} and recall that $\gbsc \subset \dme$, thanks to Assumption \ref{assu:d} and $\gbs \in \FF_{\mathfrak{sh}},$ cf.\,\eqref{f_s}.
Finally, letting $\e \searrow 0$ in  \eqref{eq_d-ee}, where we use \eqref{pw}, \eqref{conv_g_e} and \eqref{eq_d-ee0}, results in 
 \begin{equation}\label{eq_d-ee00}
  ( -\frac{1}{\eps}H_\eps'(\bar g_\e)p_\e \bar y_\e,\phi)_{L^2(D\setminus E)}
  \to \alpha \int_{\gbs=0} \frac{\chi_{D \setminus \bar E}  \phi}{|\nabla \gbs|}\,d \xi =\alpha \delta_{\gbs}(\phi) \ee  \end{equation}
for all $\phi \in H^s(D \setminus \bar E )$.
The desired inclusion \eqref{grad} is a consequence of  \eqref{eq_d-ee0}, \eqref{conv_y_e}, \eqref{pw}, and the definition of $\QQ,$ see \eqref{q}. We note that this kind of definition can be found  in \cite{bt}.

(iv) \textit{Sign condition for the state and adjoint state.}  By standard arguments \cite{troe}, we see that the equation 
\begin{equation}\label{eq_yde}
  \begin{aligned}
   -\laplace \psi + \beta(\psi)+\frac{1}{\e}H_\e(\gbe)\psi&=\chi_E(-\laplace y_d+\beta(y_d))+\chi_{D \setminus E}\beta(0)+\e\gbe \quad \text{a.e.\,in }D,
   \\\psi&=0\quad \text{on } \partial D,
 \end{aligned}
\end{equation}admits a unique solution $y_{d,\e} \in H_0^1(D) \cap H^2(D);$ cf.\,also the proof of \cite[Lem.\,4.5]{p1}.  Moreover, by using the monotony of $\beta$, the non-negativity of $H_\e$, and \eqref{conv_g_e}, one has 
\[\|y_{d,\e}\|_{H_0^1(D)}\leq c,\]where $c>0$ is independent of $\e.$
Hence, we can extract a subsequence so that \[y_{d,\e}\weakly \widetilde y_d \quad \text{ in }H_0^1(D) \ee.\] By arguing as in the step (II).(i) of the proof, we deduce that
 $\widetilde y_d$ solves 
\begin{equation}\label{eq_yd00}
  \begin{aligned}
   -\laplace \psi + \beta(\psi)&=\chi_E(-\laplace y_d+\beta(y_d))+\chi_{\O_{\gbs} \setminus E}\beta(0) \quad \text{a.e.\,in }\O_{\gbs},
   \\\psi&=0\quad \text{on } \partial \O_{\gbs},
 \end{aligned}
\end{equation}
Thus, $\widetilde y_d=y_{d,0},$ where $y_{d,0} \in H_0^1(D)$ is the extension by zero of $y_d$ and we have 
\begin{equation}\label{yde}
y_{d,\e}\weakly  y_{d,0}  \quad \text{ in }H_0^1(D) \ee.\end{equation}
Let us suppose that \eqref{f01} is true in the standing Assumption \ref{yd_f0}. By comparing \eqref{eq_yde} and the state equation associated to $\gbe$, we see that
\[\yy_\e \geq y_{d,\e} \quad \ae  D.\]
Letting $\e \searrow 0$ in the above inequality, where we employ \eqref{conv_y_e} and \eqref{yde}, results in 
 \[\ybs \geq y_{d,0} \quad \ae  D.\]
A comparison principle employed in \eqref{adj_s0} along with the fact that $\zeta \geq 0$ then yields 
\eqref{p>=0}. The assertion \eqref{p<=0} follows by the exact same arguments. This completes the proof.
\end{proof}

\begin{remark}\label{rem:conv}
When it comes to the converging term in \eqref{eq_d-ee0}, the existing literature only seems to address  the special case $\gbe=\gbs$; this may be approached in at least three alternative ways: by means of distribution theory, by employing the divergence theorem or 
in the same way we approached it \cite[Prop.\,2.2]{cm}. However, when the Dirac sequence  $\{H_\e'\}$ acts on a mapping that depends on $\e$ as well, in our case $\gbe$, the aforementioned methods fail. It seems necessary to prove that the approximating curves associated to $\gbe$ (called $\gbec$ in the previous section) have the same regularity properties as their limit $\gbsc$ (Proposition \ref{prop:per}). The purpose of the entire section \ref{sec:grad} was to gain insight into their convergence behaviour. The findings there ultimately led to the essential Corollary \ref{int_conv0}, 
which, together with the co-area formula, is the key tool for the passage to the limit  \eqref{eq_d-ee0}.
\end{remark}
\begin{remark}\label{rem:p}
Notice that in the proof of Theorem \ref{thmm} we did not pass to the limit in the strong stationary system from Theorem \ref{thm}, but in the one from Lemma \ref{lem:tool}. The reason for this is that there is no need to make use of the sign condition  \eqref{pe} that comes along with strong stationarity. If we do this, the respective limit equation reads
\begin{equation}\label{eq:pp} 
p\leq 0 \quad \text{a.e.\,in }\widetilde D_n,
\end{equation}
where  $\widetilde D_n$ is a measurable subset of  $\{x \in D: \yy_\mathfrak{sh} (x) \in \NN\}.$
This assertion is however more or less void, as $\widetilde D_n$ may have measure zero (e.g.\,if $\NN=\{z\}$ and $\{\yy_\mathfrak{sh}=z\}$ has positive measure while $\yy_\e \not = z$ a.e.\,in $D$ and for all $\e$).
However, as it turns out, this passage to the limit is not necessary at all. When we work under the assumptions from Theorem \ref{thm}, the second alternative requirement \eqref{f00} in the standing Assumption \ref{yd_f0} must hold true  (Remark \ref{rem}), that is, we arrive at \eqref{p<=0}. This already contains the sign condition for $p$ that 
 one would naively expect to get when $\e \searrow 0.$
\end{remark}

Now, we  return to our  non-smooth shape optimization problem which we recall here:

\begin{equation}\tag{$P_\O$}
 \left.
 \begin{aligned}
  \min_{\O \in \OO, E \subset \O} \quad & \int_E (y_{\O}(x)-y_d(x))^2 \; dx+\alpha\,\int_{\O}\,dx, 
\\     \text{s.t.} \quad & 
  \begin{aligned}[t]
    -\laplace y_{\O} + \beta(y_{\O})&=f \quad \text{a.e.\,in }\O,
   \\y_{\O}&=0  \quad \text{on } \partial \O.\end{aligned} \end{aligned}
 \quad \right\}
\end{equation}
The correlation between the optimal shapes of \eqref{p_sh} and the global minimizers of \eqref{p_shh} is contained in the following 
\begin{proposition}[{\cite[Prop.\,2.11]{p1}}]\label{rem:equiv}
Let  $\O^\star \in \OO$ be an optimal shape of \eqref{p_sh}. Then, each of the  functions  $g^\star \in \FF_{\mathfrak{sh}}$ that satisfy $\O_{g^\star}=\O^\star$ is a global minimizer of \eqref{p_shh}. Conversely, if $g^\star \in \FF_{\mathfrak{sh}}$ minimizes \eqref{p_shh}, then the component of $\O_{g^\star}$ that contains $E$ is an optimal shape for \eqref{p_sh}.
\end{proposition}

Thus, we may transfer the result in Theorem \ref{thmm}  to the shape optimization problem \eqref{p_sh}. In view of Proposition \ref{rem:equiv}, the optimality system associated to an optimal shape of \eqref{p_sh} is given by

\begin{theorem}[Optimality system for the optimal shape]\label{thmm0}
Let  $\O^\star \in \OO$ be an optimal shape of \eqref{p_sh} with associated state $y_{\O^\star}\in H^1_0(\O^\star) \cap H^2(\O^\star)$.  If {Assumptions \ref{gbs} and  \ref{assu:d} hold  true}, then there exists an adjoint state $p_{\O^\star}  \in  H^1_0(\O^\star)\cap H^2(\O^\star)$ and a multiplier $\zeta \in \li$ such that 
\begin{subequations} \label{eq:thmm00}  \begin{gather}
-\laplace p_{\O^\star}+\zeta p_{\O^\star}=2\chi_E(y_{\O^\star}-y_d) \text{ a.e.\ in }\O^\star, \quad p_{\O^\star}=0 \text{ on }\partial \O^\star, \label{adj_s00}
\\\zeta(x) \in [\min\{\beta_-'(y_{\O^\star}(x)),\beta_+'(y_{\O^\star}(x))\},\max\{\beta_-'(y_{\O^\star}(x)),\beta_+'(y_{\O^\star}(x))\}] \text{ a.e.\,in } {\O^\star},\label{clarke_s0}
\\\zeta(x) \in [\min\{\beta_-'(0),\beta_+'(0)\},\max\{\beta_-'(0),\beta_+'(0)\}] \text{ a.e.\,in } D \setminus {\O^\star},\label{clarke_s0'}
\\{\alpha \mu_{\partial \O^\star} \in \widetilde{\mathcal{Q}}(\partial \O^\star,y_{\O^\star},p_{\O^\star})},\label{grad0} 
\end{gather}\end{subequations}
where $ \mu_{\partial \O^\star} \in C(\bar  D \setminus  E)^\star$ is defined as 
   \begin{equation}\label{delta0}
 \mu_{\partial \O^\star}(\phi):=    \int_{\partial \O^\star} \frac{  \phi}{{|\nabla \gbs|}}\,d \xi \quad \forall\,\phi \in C(\bar  D \setminus  E)\end{equation}
and 
 \begin{equation}\label{q0}
 \begin{aligned}
 \mathcal Q(\partial \O^\star,y_{\O^\star},p_{\O^\star}) :=\{&w \in C(\bar  D \setminus  E)^\star: -\frac{1}{\e}\mathfrak h_\e  \widetilde y_\e \widetilde p_\e\weakly w \text{ in }H^s( D \setminus \bar E)^\star, \\&\text{ where } \mathfrak h_\e\weakly   \mu_{\partial \O^\star} \  \text{in }H^s( D \setminus \bar E)^\star,\\& \quad  \widetilde y_\e \weakly y_{\O^\star} \  \text{in }H_0^1(D),\ \widetilde p_\e \weakly p_{\O^\star}\  \text{in }H_0^1(D) \}.\end{aligned}
 \end{equation}
Moreover, if \eqref{f01} is true in the standing Assumption \ref{yd_f0}, then
\begin{equation}\label{p>=00}
y_{\O^\star} \geq y_d \text{ in }E, \quad y_{\O^\star} \geq 0 \text{ in }\O_{\gbs} \setminus E, \quad p_{\O^\star} \geq 0  \text{ in }\O_{\gbs}.
\end{equation}Otherwise, i.e.,
if  \eqref{f00} is true, then
\begin{equation}\label{p<=00}
y_{\O^\star} \leq y_d \text{ in }E, \quad y_{\O^\star} \leq 0 \text{ in }\O_{\gbs} \setminus E, \quad p_{\O^\star} \leq 0   \text{ in }\O_{\gbs}.
\end{equation}
\end{theorem}

\begin{remark}Let us give some comments regarding the necessary optimality conditions from Theorem \ref{thmm0}.

Clearly, \eqref{adj_s00} is the classical adjoint equation one would get in the smooth case too, see e.g.\, \cite{ks_diss}. The relations \eqref{clarke_s0}-\eqref{clarke_s0'} are the strongest one could expect, given the fact that  they are obtained by an approximation procedure; note that these are equivalent to 
\[\zeta(x) \in \partial_{\circ} \beta(y_{\O^\star}(x)) \quad \ae D,\]where $ \partial_{\circ} \beta$ stands for the Clarke subdifferential; cf.\,\cite[Thm.\,7.3.12]{schirotzek} and Assumption \ref{assu:reg0}. The sign condition for the adjoint state in \eqref{p<=00} is reasonable and entails more information than expected. For more details with respect to the relations \eqref{p>=00}-\eqref{p<=00}, see Remark \ref{rem:p}.


Concerning \eqref{grad0}, we remark that $ \mu_{\partial \O^\star}$ is a finite regular  measure concentrated on $\partial \O^\star$. Looking at \eqref{ham}, we see that the denominator appearing in the integral in \eqref{delta0}, i.e., $\nabla \gbs$, has the following geometrical significance: it describes the speed along the curve $\partial \O^\star$, see for instance \cite[Def.\,1.2.3]{pressley}.
Unfortunately, we were not able to obtain a more concrete relation in \eqref{grad0}. This is because of  the presence of the term $\frac{1}{\eps}$ and the lack of uniform  convergence of the sequences $\{\yy_\e\}$  and $\{p_\e\}$, which did not allow us 
to obtain a limit for the sequence $\{ -\frac{1}{\eps}H_\eps'(\bar g_\e)p_\e \bar y_\e\}.$ Recall that \eqref{eq_d-ee00} is  just a consequence of the fact that all the other  terms in \eqref{eq_d-ee} converge towards known values.
Looking at optimality systems from the literature \cite[Sec.\,3.3-3.4]{ks_diss}, we think that \[( -\frac{1}{\eps}H_\eps'(\bar g_\e)p_\e \bar y_\e,\phi)_{L^2(D\setminus E)}\to \int_{\gbsc} \nabla \ybs \nabla p   \frac{ \phi }{|\nabla \gbs|}\,d \xi \quad \forall\,\phi \in H^s(D \setminus \bar E) \ee,\] so that, in view of \eqref{eq_d-ee0}, \eqref{pw} and \eqref{conv_g_e}, the inclusion \eqref{grad0} is replaced by 
\begin{equation}\label{..}
\alpha \mu_{\partial \O^\star}+\int_{\partial \O^\star}  \nabla y_{ \O^\star} \nabla p_{ \O^\star}   \frac{ \phi }{|\nabla \gbs|}\,d \xi=0.\end{equation}  Then, 
in the  case that $\beta$ is a differentiable mapping, \eqref{eq:thmm00} (with \eqref{..} instead of \eqref{grad0}) would correspond to the optimality system associated to the optimal shape of a shape optimization problem governed by a smooth PDE \cite[Sec.\,3.4]{ks_diss}. 
\end{remark}

\begin{remark}
To the best of our knowledge, \eqref{eq:thmm00} is the first optimality system obtained for the control (optimal shape) of a  shape optimization problem with non-smooth PDE constraints. For  optimal design problems governed by smooth  PDEs optimality systems were derived in  \cite{oc_t} in the linear case, though more general situations can be considered there, as long as the state equation preserves its differentiability properties. This contribution deals with the more difficult case of Neumann boundary conditions and employs the functional variational approach. The shape optimization problem from \cite{oc_t}  is tackled in a direct manner, without resorting to approximating control problems such as \eqref{p10}. We point out  that this way of handling \eqref{p_sh} does not work in the present paper because of the lack of smoothness. There are of course other contributions dealing with optimality systems for the control of shape optimization problems where the state equation is a smooth PDE, see for instance \cite[Sec 3.3]{ks_diss}, \cite{k_sturm, kun, kk} and the references therein. In these works, the variational approach is purely  geometrical  and the aim is often to compute the shape derivative of the objective without the need of investigating the differentiability of the control to state map.\end{remark}

\section*{Acknowledgment}
This work was supported by the DFG grant BE 7178/3-1 for the project "Optimal Control of Viscous
Fatigue Damage Models for Brittle Materials: Optimality Systems".  
\\The author would like to thank Dr.\,Dan Tiba (IMAR Bucharest) for various fruitful discussions on the topic of shape optimization.
\bibliographystyle{plain}
\bibliography{strong_stat_coupled_pde}

\end{document}